\documentclass[12pt,reqno]{amsart}
\usepackage{amsmath, amsthm, amssymb}

\advance \topmargin by -\headheight
\advance \topmargin by -\headsep
     
\evensidemargin \oddsidemargin
\newtheorem{theorem}{Theorem}[section]
\newtheorem{lemma}[theorem]{Lemma}
\newtheorem{corollary}[theorem]{Corollary}

\theoremstyle{definition}

\theoremstyle{remark}

\numberwithin{equation}{section}

\newcommand{\mmod}[1]{\,\,(\text{mod}\,\,#1)}

\def\bfa{{\mathbf a}}

 \def\bfe{{\mathbf e}}
\def\bfg{{\mathbf g}}
\def\bfh{{\mathbf h}}

\def\bfm{{\mathbf m}}
\def\bfn{{\mathbf n}}

\def\bfu{{\mathbf u}}
\def\bfv{{\mathbf v}}
\def\bfw{{\mathbf w}}
\def\bfx{{\mathbf x}}
\def\bfy{{\mathbf y}}
\def\bfz{{\mathbf z}}

\def\calA{{\mathcal A}}  

\def\calB{{\mathcal B}} 
\def\calC{{\mathcal C}} 

\def\calD{{\mathcal D}}

\def\calI{{\mathcal I}}
\def\calJ{{\mathcal J}}

\def\calM{{\mathcal M}}

\def\calP{{\mathcal P}}

\def\calS{{\mathcal S}}
\def\calT{{\mathcal T}}

\def\Gtil{\widetilde G}

\def\dbC{{\mathbb C}}\def\dbF{{\mathbb F}}\def\dbN{{\mathbb N}}
\def\dbR{{\mathbb R}}
\def\dbZ{{\mathbb Z}}

\def\grf{{\mathfrak f}}\def\grF{{\mathfrak F}}
\def\grG{{\mathfrak G}}
\def\grJ{{\mathfrak J}}

\def\grm{{\mathfrak m}}\def\grM{{\mathfrak M}}
\def\grS{{\mathfrak S}}

\def\alp{{\alpha}} \def\bfalp{{\boldsymbol \alpha}}
\def\bet{{\beta}}  \def\bfbet{{\boldsymbol \beta}}
\def\gam{{\gamma}} \def\Gam{{\Gamma}}
\def\del{{\delta}} \def\Del{{\Delta}}
\def\zet{{\zeta}} \def\bfzet{{\boldsymbol \zeta}} 
\def\bfeta{{\boldsymbol \eta}}
\def\tet{{\theta}}  \def\Tet{{\Theta}}

\def\lam{{\lambda}}  
\def\bfnu{{\boldsymbol \nu}}
\def\bfxi{{\boldsymbol \xi}}

\def\sig{{\sigma}} \def\Sig{{\Sigma}} \def\bfsig{{\boldsymbol \sig}}
\def\bftau{{\boldsymbol \tau}}
\def\Ups{{\Upsilon}} 
\def\ome{{\omega}} \def\Ome{{\Omega}}
\def\d{{\partial}}
\def\eps{\varepsilon}

\def\le{\leqslant} \def\ge{\geqslant}

\def\d{{\,{\rm d}}}
\def\ldbrack{\lbrack\;\!\!\lbrack} \def\rdbrack{\rbrack\;\!\!\rbrack}

\begin{document}
\title[Vinogradov's mean value theorem]{Vinogradov's mean value theorem\\ via efficient congruencing}
\author[Trevor D. Wooley]{Trevor D. Wooley$^*$}
\address{School of Mathematics, University of Bristol, University Walk, Clifton, 
Bristol BS8 1TW, United Kingdom}
\email{matdw@bristol.ac.uk}
\thanks{$^*$Supported by a Royal Society Wolfson Research Merit Award.}
\date{}
\begin{abstract} We obtain estimates for Vinogradov's integral which for the first time approach those conjectured to be the best possible. Several applications of these new bounds are provided. In particular, the conjectured asymptotic formula in Waring's problem holds for sums of $s$ $k$th powers of natural numbers whenever $s\ge 2k^2+2k-3$.\end{abstract}
\maketitle

\section{Introduction} Exponential sums of large degree play a prominent role in the analysis of problems spanning the analytic theory of numbers, and in consequence the estimation of their mean values is of central significance. Some seventy-five years ago, I.~M.~Vinogradov \cite{Vin1935} obtained new estimates for such mean values by exploiting the translation-dilation invariance of associated systems of Diophantine equations. Thereby, he was able to derive powerful new estimates for exponential sums going well beyond those made available via the differencing methods of Weyl and van der Corput. Decisive progress followed in such topics as Waring's problem, the zero-free region for the Riemann zeta function, and the distribution modulo $1$ of polynomial sequences (see \cite{Vin1947}, \cite{Vin1958} and \cite{Wal1963}). Following a decade or so of technical improvement, Vinogradov's mean value theorem evolved into a form little different from that familiar to present day workers, one which for problems of degree $d$ falls short of the strength expected by a factor of order $\log d$. In this paper we obtain significant improvements in estimates associated with Vinogradov's mean value theorem, coming within a stone's throw of the sharpest possible bounds. As we explain in due course, progress of a similar scale may now be realised in numerous allied problems.\par

In order to describe our conclusions, we must introduce some notation. When $k$ is a natural number and $\bfalp \in \dbR^k$, we consider the exponential sum
\begin{equation}\label{1.1}
f_k(\bfalp;X)=\sum_{1\le x\le X}e(\alp_1x+\ldots +\alp_kx^k),
\end{equation}
where $e(z)$ denotes $e^{2\pi iz}$. It follows from orthogonality that, for natural numbers $s$, the mean value
\begin{equation}\label{1.2}
J_{s,k}(X)=\int_{[0,1)^k}|f_k(\bfalp ;X)|^{2s}\d\bfalp
\end{equation}
counts the number of integral solutions of the system of equations
\begin{equation}\label{1.3}
x_1^j+\ldots +x_s^j=y_1^j+\ldots +y_s^j\quad (1\le j\le k),
\end{equation}
with $1\le x_i,y_i\le X$ $(1\le i\le s)$. Motivated by a heuristic application of the circle method, it is widely expected that whenever $\eps>0$, one should have\footnote{Here and throughout, implicit constants in Vinogradov's notation $\ll$ and $\gg$ depend at most on $s$, $k$ and $\eps$, unless otherwise indicated.}
\begin{equation}\label{1.4}
J_{s,k}(X)\ll X^\eps (X^s+X^{2s-\frac{1}{2}k(k+1)}).
\end{equation}
Indeed, the discussion surrounding \cite[equation (7.5)]{Vau1997} supplies an $\eps$-free version of such a conjecture for $k>2$. The corresponding lower bound
\begin{equation}\label{1.5}
J_{s,k}(X)\gg X^s+X^{2s-\frac{1}{2}k(k+1)},
\end{equation}
meanwhile, is easily established (see \cite[equation (7.4)]{Vau1997}). The main conclusion of this paper, the proof of which we complete in \S7, is that the estimate (\ref{1.4}) holds whenever $s\ge k(k+1)$.

\begin{theorem}\label{theorem1.1} Suppose that $s$ and $k$ are natural numbers with $k\ge 2$ and $s\ge k(k+1)$. Then for each $\eps>0$, one has $J_{s,k}(X)\ll X^{2s-\frac{1}{2}k(k+1)+\eps}$.
\end{theorem}

If valid, the conjectured bound (\ref{1.4}) would imply a conclusion of the same shape as that of Theorem \ref{theorem1.1} provided only that $s\ge \frac{1}{2}k(k+1)$. In some sense, therefore, Theorem \ref{theorem1.1} comes within a factor $2$ of the best possible result of its type. For additive Diophantine systems of large degree $k$, this is the first occasion on which a conclusion so close to the best possible has been established, for in all previous results one misses the conjectured bounds by a factor of order $\log k$.\par

A comparison with previous results on Vinogradov's mean value theorem deserves a discussion in two parts. The original method of Vinogradov \cite{Vin1935} for estimating $J_{s,k}(X)$ was refined by means of the $p$-adic argument of Linnik \cite{Lin1943}, and achieved its most polished form in the work of Karatsuba \cite{Kar1973} and Stechkin \cite{Ste1975}. Thus, for each natural number $s$ with $s\ge k$, one has a bound of the shape
\begin{equation}\label{1.7}
J_{s,k}(X)\le D(s,k)X^{2s-\frac{1}{2}k(k+1)+\eta_{s,k}},
\end{equation}
where $D(s,k)$ is independent of $X$, and $\eta_{s,k}={\textstyle{\frac{1}{2}}}k^2(1-1/k)^{[s/k]}\le k^2e^{-s/k^2}$. For large integers $k$, the exponent $\eta_{s,k}$ is appreciably smaller than $1/k$ as soon as $s\ge 3k^2(\log k+\log \log k)$. When $s$ is sufficiently large in terms of $k$, this observation permits the proof of an asymptotic formula of the shape
\begin{equation}\label{1.8}
J_{s,k}(X)\sim \calC(s,k)X^{2s-\frac{1}{2}k(k+1)},
\end{equation}
wherein $\calC(s,k)$ is a positive number depending at most on $s$ and $k$. Note that the positivity of $\calC(s,k)$ is a consequence of the lower bound (\ref{1.5}). Let $V(k)$ denote the least natural number for which the anticipated relation (\ref{1.8}) holds. Then this classical version of Vinogradov's mean value theorem leads to the upper bound $V(k)\le 3k^2(\log k+O(\log \log k))$ (see \cite[Theorem 7.4]{Vau1997}).\par

The author's thesis work \cite{Woo1992a}, \cite{Woo1992} on repeated efficient differencing methods led to sizeable improvements in the conclusions reported in the last paragraph. Roughly speaking, the upper bound (\ref{1.7}) was established with $\eta_{s,k}\approx k^2e^{-2s/k^2}$ for $s\le k^2\log k$, and with $\eta_{s,k}\approx (\log k)^4e^{-3s/(2k^2)}$ for $s>k^2\log k$ (see \cite[Theorem 1.2]{Woo1992} for a precise statement). In the range critical in applications, the rate of decay of $\eta_{s,k}$ with respect to $s$ stemming from this progress is twice that previously available. As a consequence of these developments, we established that $V(k)\le k^2(\log k+2\log \log k+O(1))$ (see \cite{Woo1996}). We are now able to improve matters significantly.\par

Define the singular series
\begin{equation}\label{1.9b}
\grS(s,k)=\sum_{q=1}^\infty \underset{(a_1,\ldots ,a_k,q)=1}{\sum_{a_1=1}^q\dots \sum_{a_k=1}^q}\Bigl|q^{-1}\sum_{r=1}^qe((a_1r+\ldots +a_kr^k)/q)\Bigr|^{2s}
\end{equation}
and the singular integral
\begin{equation}\label{1.9c}
\grJ(s,k)=\int_{\dbR^k}\Bigl| \int_0^1 e(\bet_1\gam +\ldots +\bet_k\gam^k)\d\gam 
\Bigr|^{2s}\d\bfbet .
\end{equation}
It transpires that the positive number $\calC(s,k)$ occurring in the putative asymptotic formula (\ref{1.8}) is then given by $\calC(s,k)=\grS(s,k)\grJ(s,k)$. In \S9 we establish the asymptotic formula (\ref{1.8}) for $s\ge k^2+k+1$.

\begin{theorem}\label{theorem1.2}
When $k\ge 3$, one has $V(k)\le k^2+k+1$.
\end{theorem}

The lower bound (\ref{1.5}) implies that the asymptotic formula (\ref{1.8}) cannot hold for $s<\frac{1}{2} k(k+1)$. The condition on $s$ imposed in Theorem \ref{theorem1.2} is therefore only a factor $2$ away from the best possible conclusion of its type.\par

The estimate recorded in Theorem \ref{theorem1.1} also leads to improvements in available bounds relating to Tarry's problem. When $h$, $k$ and $s$ are positive integers with $h\ge 2$, consider the Diophantine system
\begin{equation}\label{1.9}
\sum_{i=1}^sx_{i1}^j=\sum_{i=1}^sx_{i2}^j=\ldots =\sum_{i=1}^sx_{ih}^j\quad (1\le j\le k).
\end{equation}
Let $W(k,h)$ denote the least natural number $s$ having the property that the simultaneous equations (\ref{1.9}) possess an integral solution $\bfx$ with
$$\sum_{i=1}^sx_{iu}^{k+1}\ne \sum_{i=1}^sx_{iv}^{k+1}\quad (1\le u<v\le h).$$
The problem of estimating $W(k,h)$ was investigated extensively by E. M. Wright and L.-K. Hua (see \cite{Hua1938}, \cite{Hua1949a}, \cite{Wri1948}), and very recently upper bounds for $W(k,h)$ have played a role in work of Croot and Hart \cite{CH2010} on the sum-product conjecture. L.-K. Hua was able to show that $W(k,h)\le k^2(\log k+O(1))$ for $h\ge 2$, a conclusion improved by the present author when $h=2$ with the bound $W(k,2)\le \frac{1}{2}k^2(\log k+\log \log k+O(1))$ (see \cite[Theorem 1]{Woo1996}). We improve both estimates in \S9.

\begin{theorem}\label{theorem1.3}
When $h$ and $k$ are natural numbers with $h\ge 2$ and $k\ge 2$, one has $W(k,h)\le k^2+k-2$.
\end{theorem}

Next we discuss the asymptotic formula in Waring's problem. When $s$ and $k$ are natural numbers, we denote by $R_{s,k}(n)$ the number of representations of the natural number $n$ as the sum of $s$ $k$th powers of positive integers. A heuristic application of the circle method suggests that for $k\ge 3$ and $s\ge k+1$, one should have
\begin{equation}\label{1.10}
R_{s,k}(n)=\frac{\Gam (1+1/k)^s}{\Gam(s/k)}\grS_{s,k}(n)n^{s/k-1}+o(n^{s/k-1}),
\end{equation}
where
\begin{equation}\label{1.10a}
\grS_{s,k}(n)=\sum_{q=1}^\infty \sum^q_{\substack{a=1\\ (a,q)=1}}\Bigl( q^{-1}\sum_{r=1}^qe(ar^k/q)\Bigr)^se(-na/q).
\end{equation}
Under modest congruence conditions, one has $1\ll\grS_{s,k}(n)\ll n^\eps$, and thus the conjectural relation (\ref{1.10}) may be interpreted as an honest asymptotic formula (see \cite[\S\S4.3, 4.5 and 4.6]{Vau1997} for details). Let $\Gtil(k)$ denote the least integer $t$ with the property that, for all $s\ge t$, and all sufficiently large natural numbers $n$, one has the asymptotic formula (\ref{1.10}). As a consequence of Theorem \ref{theorem1.1}, we derive the new upper bound for $\Gtil(k)$ presented in the following theorem.

\begin{theorem}\label{theorem1.4}
When $k\ge 2$, one has $\Gtil(k)\le 2k^2+2k-3$.
\end{theorem}

The first to obtain a bound for $\Gtil(k)$ were Hardy and Littlewood \cite{HL1922}, who established the bound $\Gtil(k)\le (k-2)2^{k-1}+5$. The sharpest bounds currently available for smaller values of $k$ are $\Gtil(k)\le 2^k$ $(k=3,4,5)$, due to Vaughan \cite{Vau1986a, Vau1986b}, and $\Gtil(k)\le \frac{7}{8}2^k$ $(k\ge 6)$, due to Boklan \cite{Bok1994}. For larger values of $k$, the story begins with Vinogradov \cite{Vin1935}, who showed that $\Gtil(k)\le 183k^9(\log k+1)^2$. By 1949, Hua \cite{Hua1949b} had shown that $\Gtil(k)\le (4+o(1))k^2\log k$. This upper bound was improved first by the author \cite{Woo1992} to $\Gtil(k)\le (2+o(1))k^2\log k$, and most recently by Ford \cite{For1995} to $\Gtil(k)\le 
(1+o(1))k^2\log k$. The latter two authors, Parsell \cite{Par2009}, and most recently Boklan and Wooley \cite{BK2010}, have also computed explicit upper bounds for $\Gtil(k)$ when $k\le 20$. In particular, one has the bounds $\Gtil(7)\le 112$, $\Gtil(8)\le 224$ due to Boklan \cite{Bok1994}, and $\Gtil(9)\le 365$, $\Gtil(10)\le 497$, $\Gtil(11)\le 627$, $\Gtil(12)\le 771$, $\Gtil(13)\le 934$, $\Gtil(14)\le 1112$, $\Gtil(15)\le 1307$, $\Gtil(16)\le 1517$, $\Gtil(17)\le 1747$, $\Gtil(18)\le 1992$, $\Gtil(19)\le 2255$, $\Gtil(20)\le 2534$ due to Boklan and Wooley \cite{BK2010}. The conclusion of Theorem \ref{theorem1.4} supersedes all of these previous results for $k\ge 7$, establishing that $\Gtil(7)\le 109$, $\Gtil(8)\le 141$, $\Gtil(9)\le 177$, ..., $\Gtil(20)\le 837$. Furthermore, the strength of Theorem \ref{theorem1.1} opens new possibilities for transforming estimates for $J_{s,k}(X)$ into bounds for auxiliary mean values suitable for investigating Waring's problem. This is a matter that we shall pursue further elsewhere (see \cite{Woo2011}).\par
 
We turn next to estimates of Weyl type for exponential sums. Here we present conclusions of two types, one applicable to exponential sums $f_k(\bfalp;X)$ defined by (\ref{1.1}) wherein a single coefficient $\alp_j$ is poorly approximable, and a second applicable when $\bfalp$ is poorly approximable as a $k$-tuple.

\begin{theorem}\label{theorem1.5} Let $k$ be an integer with $k\ge 2$, and let $\bfalp\in \dbR^k$. Suppose that there exists a natural number $j$ with $2\le j\le k$ such that, for some $a\in \dbZ$ and $q\in \dbN$ with $(a,q)=1$, one has $|\alp_j-a/q|\le q^{-2}$ and $q\le X^j$. Then one has
$$f_k(\bfalp;X)\ll X^{1+\eps}(q^{-1}+X^{-1}+qX^{-j})^{\sig (k)},$$
where $\sig(k)^{-1}=2k(k-1)$.
\end{theorem}

We remark that the factor $X^\eps$ in the conclusion of Theorem \ref{theorem1.5} may be replaced by $\log (2X)$ if one increases $\sig(k)^{-1}$ from $2k(k-1)$ to $2k^2-2k+1$.

\begin{theorem}\label{theorem1.6} Let $k$ be an integer with $k\ge 2$, and let $\tau$ and $\del$ be real numbers with $\tau^{-1}>4k(k-1)$ and $\del>k\tau$. Suppose that $X$ is sufficiently large in terms of $k$, $\del$ and $\tau$, and further that $|f_k(\bfalp;X)|\ge X^{1-\tau}$. Then there exist integers $q,a_1,\ldots ,a_k$ such that $1\le q\le X^\del$ and $|q\alp_j-a_j|\le X^{\del -j}$ $(1\le j\le k)$.
\end{theorem}

The conclusion of Theorem \ref{theorem1.5} may be compared, for smaller exponents $k$, with Weyl's inequality (see \cite[Lemma 2.4]{Vau1997}). The latter provides an estimate of the same shape as that of Theorem \ref{theorem1.5} in the special case $j=k$, with the exponent $2^{k-1}$ in place of $2k(k-1)$. The conclusion of Theorem \ref{theorem1.5} is therefore superior to Weyl's inequality for $k\ge 8$. Subject to the condition $k\ge 6$, Heath-Brown \cite{HB1988} has shown that whenever there exist $a\in \dbZ$ and $q\in \dbN$ with $(a,q)=1$ and $|\alp -a/q|\le q^{-2}$, then one has
\begin{equation}\label{1.hb}
\sum_{1\le x\le X}e(\alp x^k)\ll X^{1-\frac{8}{3}2^{-k}+\eps} (X^3q^{-1}+1+qX^{3-k})^{\frac{4}{3}2^{-k}}.
\end{equation}
With the same conditions on $\alp$, Robert and Sargos \cite[Th\'eor\`eme 4 et Lemme 7]{RS2000} have shown that for $k\ge 8$, one has
\begin{equation}\label{1.14a}
\sum_{1\le x\le X}e(\alp x^k)\ll X^{1-3\cdot 2^{-k}+\eps}(X^4q^{-1}+1+qX^{4-k})^{\frac{8}{5}2^{-k}}.
\end{equation}
When $k\ge 9$, our conclusions in these special situations are superior to those of Heath-Brown, and those of Robert and Sargos, even for the restricted set of $\alp$ for which either (\ref{1.hb}) or (\ref{1.14a}) prove superior to Weyl's inequality. Finally, the methods of Vinogradov yield results of the type provided by Theorem \ref{theorem1.5} with the exponent $2k(k-1)$ replaced by $(C+o(1))k^2\log k$, for suitable values of $C$. For example, Linnik \cite{Lin1943} obtained the permissible value $C=22400$, Hua \cite{Hua1949b} obtained $C=4$, and the sharpest bound available hitherto, due to the author \cite{Woo1995} is tantamount to $C=3/2$. We note also that Wooley \cite{Woo1992}, Ford \cite{For1995}, Parsell \cite{Par2009}, and most recently Boklan and Wooley \cite{BK2010}, have computed explicit upper bounds for $\sig (k)$ when $k\le 20$. The conclusion of Theorem \ref{theorem1.5} is superior to these earlier numerical conclusions in all cases, and is transparently sharper for larger values of $k$ by a factor asymptotically of order $\log k$. Similar comments apply to the conclusion of Theorem \ref{theorem1.6}, a suitable reference to earlier work being \cite[Chapters 4 and 5]{Bak1986}.\par

Our final result concerns the distribution modulo $1$ of polynomial sequences. Here, 
we write $\|\tet\|$ for $\underset{y\in \dbZ}{\min}|\tet -y|$.

\begin{theorem}\label{theorem1.7}
Let $k$ be an integer with $k\ge 2$, and define $\tau(k)$ by $\tau(k)^{-1}=4k(k-1)$. Then whenever $\bfalp\in\dbR^k$ and $N$ is sufficiently large in terms of $k$ and $\eps$, one has
$$\min_{1\le n\le N}\| \alp_1n+\alp_2n^2+\ldots +\alp_kn^k\|<N^{\eps -\tau(k)}.$$
\end{theorem}

For comparison, R. C. Baker \cite[Theorem 4.5]{Bak1986} provides a similar conclusion in which the exponent $4k(k-1)$ is replaced by $(8+o(1))k^2\log k$, a conclusion subsequently improved by the author to $(4+o(1))k^2\log k$ (see \cite[Corollary 1.3]{Woo1992}). For smaller values of $k$, meanwhile, a conclusion similar to that of Theorem \ref{theorem1.7} is delivered by \cite[Theorem 5.2]{Bak1986}, but with the exponent $2^{k-1}$ in place of $4k(k-1)$. The conclusion of Theorem \ref{theorem1.7} is superior to these earlier results for $k\ge 10$.\par

Given the scale of the improvement in estimates made available via Theorem \ref{theorem1.1}, it is natural to enquire whether it is now possible to derive visible improvements in the zero-free region for the Riemann zeta function. The estimate supplied by Theorem \ref{theorem1.1} has the shape $J_{s,k}(X)\le D(k,\eps)X^{2s-\frac{1}{2}k(k+1)+\eps}$ for $s\ge k(k+1)$, and the nature of the quantity $D(k,\eps)$ plays a critical role in determining the rate of growth of $|\zet(\sig +it)|$ with respect to $t$ when $\sig$ is close to $1$. It seems clear that, while some numerical improvement will be made available via the methods underlying Theorem \ref{theorem1.1}, such improvements will not lead to asymptotically significant improvements in the zero-free region. We refer the reader to the work of Ford \cite{For2002} for a discussion of recent numerical improvements to which our new results may be expected to contribute.\par

The arguments that underly our proof of Theorem \ref{theorem1.1}, which in a nod to the earlier use of {\it efficient differencing} we refer to loosely as {\it efficient congruencing} methods, change little when the setting for Vinogradov's mean value theorem is shifted from $\dbZ$ to the ring of integers of a number field. A significant feature of our estimates in this respect is that when $s\ge k(k+1)$, one is at most a factor $X^\eps$ away from the truth. In common with Birch's application \cite{Bir1961} of Hua's lemma in number fields, this aspect of our estimates makes them robust to variation in the degree of the field extension, since the strength of corresponding Weyl-type estimates for exponential sums no longer plays a significant role in applications. Thus, in any number field, one is able to establish the validity of the Hasse Principle, and of Weak Approximation, for diagonal equations of degree $d$ in $2d^2+2d+1$ or more variables, and moreover one is able to obtain the expected density of rational solutions of such equations. Hitherto, such a conclusion was available via the methods of Birch \cite{Bir1961} only for diagonal forms of degree $d$ in $2^d+1$ or more variables. In a similar manner, the robustness of the efficient congruencing method permits conclusions to be drawn over function fields, such as $\dbF_q(t)$, matching in strength what is to be found within this paper. We intend to record the consequences of our methods for such problems in forthcoming work.\par

Finally, the efficient congruencing operation may be applied with success in a number of multidimensional problems related to Vinogradov's mean value theorem. Thus, the work of Arkhipov, Chubarikov and Karatsuba, and Parsell, on exponential sums in many variables (see \cite{ACK} and \cite{Par2005}, for example) may be improved in a manner no less dramatic than can be seen in the context of the version of Vinogradov's mean value theorem described within this paper. This again is a topic to which we intend to return elsewhere.\par

The methods underlying our proof of Theorem \ref{theorem1.1} are complicated by the need to control non-singularity constraints modulo various powers of a prime, and this control must be exercised within an iterative process a step ahead of its application. This and other complicating factors obscure the key ideas of our argument, and so we have taken the liberty of providing, in \S2 below, a sketch of the fundamental efficient congruencing process. The reader will also find there an outline of the classical approach of Vinogradov, together with the repeated efficient differencing process. Next, in \S3, we prepare the notation and basic notions required in our subsequent deliberations. The concern of \S4 is an estimate for a system of basic congruences, and \S5 describes the conditioning process required to guarantee appropriate non-singularity conditions in subsequent steps of the efficient congruencing argument. In \S6 we discuss the efficient congruencing process itself. We combine these tools in \S7 with an account of the iterative process, ultimately delivering Theorem \ref{theorem1.1}. In \S8 we turn to our first applications, with the proof of Theorems \ref{theorem1.5}, \ref{theorem1.6} and \ref{theorem1.7}. Next, in \S9, we consider Tarry's problem, and establish Theorems \ref{theorem1.2} and \ref{theorem1.3}. Finally, in \S10, we consider the asymptotic formula in Waring's problem, and establish Theorem \ref{theorem1.4}. We finish in \S11 by describing a heuristic argument which implies the best possible bound of the shape (\ref{1.4}).\par

\section{A sketch of the efficient congruencing process} 
Our goal in this section is to offer an indication of the strategy underlying the efficient congruencing process key to our new bounds. At the same time, it is expedient to introduce some notation of use throughout this paper. In what follows, the letter $k$ denotes a fixed integer exceeding $1$, the letter $s$ will be a positive integer, and $\eps$ denotes a sufficiently small positive number. Since $k$ is considered fixed, we usually find it convenient to drop explicit mention of $k$ from the exponential sum $f_k(\bfalp;X)$ and its mean value $J_{s,k}(X)$, defined in (\ref{1.1}) and ({\ref{1.2}), respectively. We take $X$ to be a large real number depending at most on $k$, $s$ and $\eps$, unless otherwise indicated. In an effort to simplify our analysis, we adopt the following convention concerning the number $\eps$. Whenever $\eps$ appears in a statement, either implicitly or explicitly, we assert that the statement holds for each $\eps>0$. Note that the ``value'' of $\eps$ may consequently change from statement to statement. Finally, we make use of vector notation in a slightly unconventional manner. Thus, we may write $a\le \bfz\le b$ to denote that $a\le z_i\le b$ for $1\le i\le t$, we may write $\bfz\equiv \bfw\pmod{p}$ to denote that $z_i\equiv w_i\pmod{p}$ $(1\le i\le t)$, or on occasion $\bfz\equiv \xi\pmod{p}$ to denote that $z_i\equiv \xi\pmod{p}$ $(1\le i\le t)$. Confusion should not arise if the reader interprets similar statements in like manner.\par

We refer to the exponent $\lam_s$ as {\it permissible} when, for each positive number $\eps$, and for any real number $X$ sufficiently large in terms of $s$, $k$ and $\eps$, one has $J_s(X)\ll X^{\lam_s+\eps}$. Define $\lam_s^*$ to be the infimum of the set of exponents $\lam_s$ permissible for $s$ and $k$, and then put $\eta_s=\lam_s^*-2s+\frac{1}{2}k(k+1)$. Thus, whenever $X$ is sufficiently large in terms of $s$, $k$ and $\eps$, one has
\begin{equation}\label{2.01}
J_s(X)\ll X^{\lam_s^*+\eps},
\end{equation}
where
\begin{equation}\label{2.02}
\lam_s^*=2s-{\textstyle{\frac{1}{2}}}k(k+1)+\eta_s.
\end{equation}
Note that, in view of the lower bound (\ref{1.5}) and the trivial estimate $J_s(X)\le X^{2s}$, one has $0\le \eta_s\le {\textstyle{\frac{1}{2}}}k(k+1)$ for $s\in \dbN$. Vinogradov's method employs the translation-dilation invariance of the system (\ref{1.3}) to bound $\eta_{s+k}$ in terms of $\eta_s$ by efficiently engineering a strong congruence condition on the variables.\par

After Linnik \cite{Lin1943}, the classical approach to Vinogradov's mean value theorem imposes an initial congruence condition on the variables of the system (\ref{1.3}) by dividing into congruence classes modulo $p$, for a suitably chosen prime $p$. Let $\tet$ be a positive number with $0<\tet \le 1/k$, and consider a prime number $p$ with $X^\tet <p\le 2X^\tet$. The existence of such a prime is guaranteed by the Prime Number Theorem, or indeed by weaker results such as Bertrand's Postulate. Next, when $c$ and $\xi$ are non-negative integers, and $\bfalp \in [0,1)^k$, define
\begin{equation}\label{2.2}
\grf_c(\bfalp ;\xi)=\sum_{\substack{1\le x\le X\\ x\equiv \xi\mmod{p^c}}}e(\psi(x;\bfalp)),
\end{equation}
where
\begin{equation}\label{2.3}
\psi(x;\bfalp)=\alp_1x+\alp_2x^2+\ldots +\alp_kx^k.
\end{equation}
An application of H\"older's inequality now leads from (\ref{1.1}) to the bound
\begin{equation}\label{2.4}
|f(\bfalp;X)|^{2s}=\Bigl| \sum_{\xi=1}^{p^c}\sum_{\substack{1\le x\le X\\ x\equiv \xi\mmod{p^c}}}e(\psi(x;\bfalp))\Bigr|^{2s}\le (p^c)^{2s-1}\sum_{\xi=1}^{p^c}|\grf_c(\bfalp ;\xi)|^{2s}.
\end{equation}

\par Let us focus now on the mean value $J_{s+k}(X)$ defined via (\ref{1.2}). In order to save clutter, when $G:[0,1)^k\rightarrow \dbC$ is measurable, we write
$$\oint G(\bfalp)\d \bfalp =\int_{[0,1)^k}G(\bfalp)\d \bfalp .$$
On substituting (\ref{2.4}) into the analogue of (\ref{1.2}) with $s$ replaced by $s+k$, we find that
\begin{equation}\label{2.5}
J_{s+k}(X)\ll X^{2s\tet}\max_{1\le \xi\le p}\oint |f(\bfalp ;X)^{2k}\grf_1(\bfalp ;\xi)^{2s}|\d \bfalp .
\end{equation}
The mean value on the right hand side of (\ref{2.5}) counts the number of integral solutions of the system
\begin{equation}\label{2.6}
\sum_{i=1}^k(x_i^j-y_i^j)=\sum_{l=1}^s((pu_l+\xi)^j-(pv_l+\xi)^j)\quad (1\le j\le k),
\end{equation}
with $1\le \bfx,\bfy\le X$ and $(1-\xi)/p\le \bfu,\bfv\le (X-\xi)/p$. But, as a consequence of the Binomial Theorem, the validity of the equations (\ref{2.6}) implies that
\begin{equation}\label{2.7}
\sum_{i=1}^k((x_i-\xi)^j-(y_i-\xi)^j)=p^j\sum_{l=1}^s(u_l^j-v_l^j)\quad (1\le j\le k),
\end{equation}
whence
\begin{equation}\label{2.8}
\sum_{i=1}^k(x_i-\xi)^j\equiv \sum_{i=1}^k(y_i-\xi)^j\pmod{p^j}\quad (1\le j\le k).
\end{equation}

\par The congruences (\ref{2.8}) provide the {\it efficient} congruence condition mentioned earlier, with the artificial condition modulo $p$ imposed via (\ref{2.4}) converted into a system of congruence conditions modulo $p^j$ for $1\le j\le k$, as opposed merely to a system of congruence conditions modulo $p$. Suppose that $\bfx$ is {\it well-conditioned}, by which we mean that $x_1,\dots ,x_k$ lie in distinct congruence classes modulo $p$. Then, given an integral $k$-tuple $\bfn$, the solutions of the system
$$\sum_{i=1}^k(x_i-\xi)^j\equiv n_j\pmod{p}\quad (1\le j\le k),$$
with $1\le \bfx\le p$, may be lifted uniquely to solutions of the system
$$\sum_{i=1}^k(x_i-\xi)^j\equiv n_j\pmod{p^k}\quad (1\le j\le k),$$
with $1\le \bfx\le p^k$. In this way, the congruences (\ref{2.8}) essentially imply that
\begin{equation}\label{2.9}
\bfx\equiv \bfy\pmod{p^k},
\end{equation}
provided that we inflate our estimates by the combinatorial factor $k!$ to account for the multiplicity of solutions modulo $p$, together with a factor $p^{\frac{1}{2}k(k-1)}$ to account for solutions introduced as one considers for a fixed $\bfn'$ the possible choices for $\bfn\pmod{p^k}$ with $n_j\equiv n_j^\prime\pmod{p^j}$ $(1\le j\le k)$.\par

In the classical argument, one chooses $\tet=1/k$, so that $p^k>X$. Since $1\le \bfx, \bfy\le X$, one is then forced to conclude from the congruential condition (\ref{2.9}) that $\bfx=\bfy$, and in (\ref{2.7}) this in turn implies that
$$\sum_{l=1}^s(u_l^j-v_l^j)=0\quad (1\le j\le k).$$
The number of solutions of this system with $(1-\xi)/p\le \bfu,\bfv\le (X-\xi)/p$ is readily seen to be $O(J_s(X/p))$, and thus the corresponding number of solutions of (\ref{2.7}) with $\bfx=\bfy$ is $O(X^kJ_s(X/p))$. Thus, in view of (\ref{2.01}), one obtains
$$J_{s+k}(X)\ll (X^\tet)^{2s+\frac{1}{2}k(k-1)}X^kJ_s(X/p)\ll (X^\tet)^{2s+\frac{1}{2}k(k-1)}X^k(X^{1-\tet})^{\lam_s^*+\eps}.$$
Since $\tet=1/k$, it follows from (\ref{2.02}) that $\lam_{s+k}^*\le 2(s+k)-{\textstyle{\frac{1}{2}}}k(k+1)+\eta_{s+k}$, with $\eta_{s+k}\le \eta_s(1-1/k)$. On recalling the estimate $J_k(X)\le k!X^k$, stemming from Newton's formulae, the classical bound (\ref{1.7}) with $\eta_s=\frac{1}{2}k^2(1-1/k)^{[s/k]}$ follows by induction.\par

Suppose now that we take $\tet<1/k$, and interpret the condition (\ref{2.9}) by defining the $k$-tuple $\bfh$ by means of the relation $\bfh=(\bfx-\bfy)p^{-k}$, so that $\bfx=\bfy+\bfh p^k$. On substituting into (\ref{2.7}), we obtain the new system
\begin{equation}\label{2.11}
\sum_{i=1}^k\Psi_j(y_i,h_i,p)=\sum_{l=1}^s(u_l^j-v_l^j)\quad (1\le j\le k),
\end{equation}
where
$$\Psi_j(y,h,p)=p^{-j}((y+hp^k-\xi)^j-(y-\xi)^j)\quad (1\le j\le k).$$
The number of solutions of the system (\ref{2.11}) subject to the associated conditions $1\le \bfy\le X$, $|\bfh|\le Xp^{-k}$ and $(1-\xi)/p\le \bfu,\bfv\le (X-\xi)/p$, may be reinterpreted by means of an associated mean value of exponential sums. An application of Schwarz's inequality bounds this mean value in terms of $J_s(X/p)$ and a new mean value that counts integral solutions of the system
\begin{equation}\label{2.12}
\sum_{i=1}^k(\Psi_j(x_i,h,p)-\Psi_j(y_i,h,p))=\sum_{l=1}^s(u_l^j-v_l^j)\quad (1\le j\le k),
\end{equation}
with variables satisfying similar conditions to those above. We now have the option of repeating the process of imposing an efficient congruence condition on $\bfx$ and $\bfy$, much as before, by pushing the variables $\bfu$ and $\bfv$ into congruence classes modulo a suitable new prime number $\varpi$. In this way, one may estimate $J_{s+k}(X)$ by iteratively bounding the number of solutions of a system of type (\ref{2.12}) by a similar one, wherein the polynomial $\Psi_j(z)=\Psi_j(z,h,p)$ is replaced for $1\le j\le k$ by one of the shape $\Phi_j(z,g,\varpi)=\varpi^{-j}(\Psi_j(z+g\varpi^k)-\Psi_j(z))$. This {\it repeated efficient differencing process}, so-called owing to its resemblance to Weyl differencing, delivers the more efficient choice of parameter $\tet\approx k/(k^2+\eta_s)$. In the most important range for $s$, one obtains an estimate roughly of the shape $\eta_{s+k}\le \eta_s(1-2k/(k^2+\eta_s))$, and this yields $\eta_s\approx k^2e^{-2s/k^2}$ (see \cite{Woo1992} for details).\par

The strategy underlying Vinogradov's method, as seen in both its classical and repeated efficient differencing formulations, is that of transforming an initial congruence condition into a differencing step, with the ultimate aim in (\ref{2.6}), and its variants such as (\ref{2.12}), of forcing $2k$ variables to obey a diagonal condition. In this paper we instead view Vinogradov's method as an efficient generator of congruence conditions. Thus, the initial condition modulo $p$ amongst $2s$ variables underlying the mean value $J_{s+k}(X)$ efficiently generates the stronger condition modulo $p^k$ visible in (\ref{2.9}). Our strategy now is to exploit this condition so as to push $2s$ variables into the same congruence class modulo $p^k$ within a new mean value, and efficiently extract from this a fresh congruence condition modulo $p^{k^2}$. By repeating this process, one extracts successively stronger congruence conditions, and these may be expected to yield successively stronger mean value estimates.\par

There is a critical detail concerning which we have, thus far, remained silent. We supposed in advance of (\ref{2.9}) that the $k$-tuple $\bfx$ was well-conditioned, and indeed similar assumptions must be made at each point of the repeated efficient differencing process. There are several possible approaches to the challenge of ensuring this well-conditioning of variables, the most straightforward being to preselect the prime so that the bulk of solutions are well-conditioned (see \cite{Woo1993} for a transparent application of this idea). The problem of ensuring well-conditioning causes considerable difficulty in the analysis of the efficient congruencing argument in this paper, for our prime is fixed once and for all at the outset of our argument. For now we ignore this complication so as to better expose the underlying ideas.\par

We now outline the repeated efficient congruencing argument. In the first instance, we take $0<\tet\le 1/k^2$. Observe that, in view of the condition (\ref{2.9}), one may derive from (\ref{2.6}) the upper bound
$$J_{s+k}(X)\ll (X^\tet)^{2s+\frac{1}{2}k(k-1)}\max_{1\le \xi\le p}\oint \Bigl( \sum_{\eta=1}^{p^k}|\grf_k(\bfalp;\eta)|^2\Bigr)^k|\grf_1(\alp;\xi)|^{2s}\d \bfalp .$$
By H\"older's inequality, therefore, one sees that
\begin{equation}\label{2.13b}
J_{s+k}(X)\ll (X^\tet)^{2s+\frac{1}{2}k(k-1)}(X^{k\tet})^k\max_{1\le \xi\le p}\max_{1\le \eta\le p^k}I(\xi,\eta),
\end{equation}
where
$$I(\xi,\eta)=\oint |\grf_k(\bfalp;\eta)^{2k}\grf_1(\bfalp;\xi)^{2s}|\d \bfalp .$$
A further application of H\"older's inequality shows that
\begin{equation}\label{2.13a}
I(\xi,\eta)\le \Bigl( \oint |\grf_1(\bfalp ;\xi)|^{2s+2k}\d \bfalp \Bigr)^{1-k/s} \Bigl( \oint |\grf_1(\bfalp;\xi)^{2k}\grf_k(\bfalp ;\eta)^{2s}|\d \bfalp \Bigr)^{k/s}.
\end{equation}

The first integral on the right hand side of (\ref{2.13a}) counts the number of integral solutions of the system
$$\sum_{i=1}^{s+k}((pu_i+\xi)^j-(pv_i+\xi)^j)=0\quad (1\le j\le k),$$
with $(1-\xi)/p\le \bfu,\bfv\le (X-\xi)/p$. An application of the Binomial Theorem shows this to be $O(J_{s+k}(X/p))$. By orthogonality, meanwhile, the second integral is bounded above by the number of solutions of the system
\begin{equation}\label{2.14}
\sum_{i=1}^k(x_i^j-y_i^j)=\sum_{l=1}^s((p^ku_l+\eta)^j-(p^kv_l+\eta)^j)\quad (1\le j\le k),
\end{equation}
with $1\le \bfx,\bfy\le X$, $\bfx\equiv \bfy\equiv \xi\pmod{p}$ and $(1-\eta)/p^k\le \bfu,\bfv\le (X-\eta)/p^k$. As in the classical treatment sketched above, it follows as a consequence of the Binomial Theorem that the validity of the equations (\ref{2.14}) implies that
$$\sum_{i=1}^k((x_i-\eta)^j-(y_i-\eta)^j)=p^{jk}\sum_{l=1}^s(u_l^j-v_l^j)\quad (1\le j\le k),$$
whence
\begin{equation}\label{2.15}
\sum_{i=1}^k(x_i-\eta)^j\equiv \sum_{i=1}^k(y_i-\eta)^j\pmod{p^{jk}}\quad (1\le j\le k).
\end{equation}

\par The system (\ref{2.15}) provides an even more efficient congruence condition than that offered by (\ref{2.8}), tempered with a slightly diminished return stemming from the fact that the $x_i$ and $y_i$ all lie in the common congruence class $\xi$ modulo $p$. On the face of it, the latter unequivocally prevents these variables being well-conditioned. However, let us assume for now that $x_1,\ldots ,x_k$ are distinct modulo $p^2$, and likewise $y_1,\ldots ,y_k$. It transpires that on this occasion, one may lift solutions modulo $p^2$ to solutions modulo $p^{k^2}$. Indeed, the congruences (\ref{2.15}) essentially imply that
\begin{equation}\label{2.16}
\bfx\equiv \bfy\pmod{p^{k^2}},
\end{equation}
provided that one inserts a compensating factor $k!(p^{k+1})^{\frac{1}{2}k(k-1)}$ into the concomitant estimates. At this point one could repeat the whole process, empoying (\ref{2.16}) to engineer a fresh congruence condition modulo $p^{k^3}$, then modulo $p^{k^4}$, and so on. However, in order to illuminate this efficient congruencing argument, we examine instead the consequences of the assumption that $\tet=1/k^2$. In such circumstances, one has $p^{k^2}>X$, and so it follows from (\ref{2.16}) that $\bfx=\bfy$. Since $\bfx\equiv \bfy\equiv \xi\pmod{p}$, the number of possible choices for $\bfx$ and $\bfy$ is $O((X/p)^k)$. Substituting into (\ref{2.14}), we deduce that
\begin{align}
\oint|\grf_1(\bfalp;\xi)^{2k}\grf_k(\bfalp;\eta)^{2s}|\d \bfalp &\ll (X^\tet)^{\frac{1}{2}k(k^2-1)}(X^{1-\tet})^k\oint |\grf_k(\bfalp;\eta)|^{2s}\d \bfalp \notag \\
&\ll (X^\tet)^{\frac{1}{2}k(k^2-1)}(X^{1-\tet})^kJ_s(X/p^k).\label{2.17}
\end{align}

\par If we now substitute (\ref{2.17}) into (\ref{2.13a}), we obtain
$$I(\xi,\eta)\ll (J_{s+k}(X/p))^{1-k/s}\left( (X^\tet)^{\frac{1}{2}k(k^2-1)}(X^{1-\tet})^kJ_s(X/p^k)\right)^{k/s},$$
whence, in view of (\ref{2.01}), it follows from (\ref{2.13b}) that
\begin{align*}
J_{s+k}(X)\ll & \,\left( (X^\tet)^{2s+2k-\frac{1}{2}k(k+1)}(X^{1-\tet})^{\lam_{s+k}^*+\eps}\right)^{1-k/s}\\
& \, \times \left( (X^\tet)^{2ks-\frac{1}{2}k(k+1)+2k+\frac{1}{2}k(k^2-1)}(X^{1-\tet})^k(X^{1-k\tet})^{\lam_s^*+\eps}\right)^{k/s}.
\end{align*}
Consequently, from (\ref{2.02}) we discern the upper bound
$$J_{s+k}(X)\ll (X^{\eta_{s+k}(1-\tet)})^{1-k/s}\left(X^{-k+k^3\tet}(X^{\eta_s(1-k\tet)})\right)^{k/s}X^{2s+2k-\frac{1}{2}k(k+1)+\eps}.$$
Recall that $\tet=1/k^2$. Since $\lam_{s+k}^*=2s+2k-\frac{1}{2}k(k+1)+\eta_{s+k}$ is an infimal exponent, it follows that for a sequence of values of $X$ tending to $\infty$, one has
$$X^{\eta_{s+k}-\eps}\ll X^\eps (X^{1-1/k^2})^{(1-k/s)\eta_{s+k}}(X^{1-1/k})^{(k/s)\eta_s},$$
whence for each positive number $\eps$, one has
$$\eta_{s+k}\le (1-k/s)(1-1/k^2)\eta_{s+k}+(k/s)(1-1/k)\eta_s+\eps .$$
Noting again the infimal definition of $\lam^*_{s+k}$, we therefore deduce that
\begin{equation}\label{2.18}
\eta_{s+k}\le \frac{(1-1/k)\eta_s}{1+(s/k-1)(1/k^2)}.
\end{equation}

Provided that $s$ is no larger than about $k^{5/2}$, a modest computation leads from the iterative relation (\ref{2.18}) to the upper bound $$\eta_{s+k}\le (1-s/k^3)\eta_s\le e^{-s/k^3}\eta_s.$$
One therefore sees that $\eta_s$ is no larger than about $k^2e^{-\frac{1}{2}(s/k^2)^2}$. By comparison with the classical bound $\eta_s\le k^2e^{-s/k^2}$ mentioned following (\ref{1.7}), one has considerable additional decay in the upper bound for $\eta_s$ as soon as $s$ is a little larger than $k^2$. Indeed, even an estimate of this quality would establish, for example, that $\Gtil (k)\ll k^2(\log k)^{1/2}$, greatly improving the bound $\Gtil(k)\le (1+o(1))k^2\log k$ due to Ford \cite{For1995}.\par

For each natural number $N$, the pursuit of an $N$-fold repeated efficient congruencing process delivers bounds with the approximate shape
$$\eta_{s+k}\le \frac{\eta_s}{1+(s/k)^N(1/k^{N+1})}.$$
When $s>k^2$, it is apparent that the upper bound on the right hand side here converges to zero as $N$ goes to infinity. Such a bound comes close to delivering Theorem \ref{theorem1.1}. Two serious obstructions remain. The first is the removal of the assumption throughout that variables are suitably well-conditioned whenever this is essential. Since our auxiliary prime number $p$ is fixed once and for all at the opening of our argument, we are forced to engineer well-conditioning directly using this single prime $p$. Such has the potential to weaken substantially our conclusions, and we are forced to consider a complex iterative process rather difficult to control. The second obstruction is less severe. The condition $s>k^2$ must be replaced by $s=k^2$, and owing to the possibility of ill-conditioned solutions, a direct approach would be successful, at best, only when $s\ge k^2+k$. Once again, therefore, we are forced to negotiate delicate issues associated with a complex iterative process.

\section{Preliminary manoeuvres} We begin in this section with some notation and definitions of use in our subsequent discussion. Let $k$ be a fixed integer with $k\ge 2$, and let $\del$ be a small positive number. We consider a natural number $u$ with $u\ge k$, and we put $s=uk$. Our goal is to show that $\lam_{s+k}^*=2(s+k)-\frac{1}{2}k(k+1)$, whence $\eta_{s+k}=0$. In view of the infimal definition of $\lam_{s+k}^*$, there exists a sequence of natural numbers $(X_n)_{n=1}^\infty$, tending to infinity, with the property that
\begin{equation}\label{3.1}
J_{s+k}(X_n)>X_n^{\lam_{s+k}^*-\del}\quad (n\in \dbN).
\end{equation}
Provided that $X_n$ is sufficiently large, we have also for $X_n^{\del^2}<Y\le X_n$ the corresponding upper bounds
\begin{equation}\label{3.2}
J_t(Y)<Y^{\lam_t^*+\del}\quad (t=s,s+k).
\end{equation}
Notice that when $s>k^2$, the trivial inequality $|f(\bfalp;X)|\le X$ leads from (\ref{1.2}) to the upper bound
$$J_{s+k}(X)\le X^{2(s-k^2)}\oint |f(\alp;X)|^{2k(k+1)}\d\bfalp \le X^{2(s-k^2)}J_{k(k+1)}(X).$$
It then follows from the above discussion that whenever $s>k^2$, one has $\eta_{s+k}\le \eta_{k(k+1)}$. With an eye toward future applications, we shall continue to consider general values of $s$ with $s\ge k^2$ until the very climax of the proof of Theorem \ref{theorem1.1}, and only at that point specialise to the situation with $s=k^2$. As we have just shown, the desired conclusion when $s>k^2$ is an easy consequence of this special case. Finally, we take $N$ to be a natural number sufficiently large in terms of $s$ and $k$, and we put $\tet=\frac{1}{2}(k/s)^{N+1}$. Note that we are at liberty to take $\del$ to be a positive number with $\del<(Ns)^{-3N}$, so that $\del$ is in particular small compared to $\tet$. We focus now on a fixed element $X=X_n$ of the sequence $(X_n)$, which we may assume to be sufficiently large in terms of $s$, $k$, $N$ and $\del$, and put $M=X^\tet$. Thus we have $X^\del < M^{1/N}$.\par

Let $p$ be a fixed prime number with $M<p\le 2M$ to be chosen in due course. That such a prime exists is a consequence of the Prime Number Theorem. We will find it necessary to consider {\it well-conditioned} $k$-tuples of integers belonging to distinct congruence classes modulo a suitable power of $p$. Denote by $\Xi_c(\xi)$ the set of $k$-tuples $(\xi_1,\ldots ,\xi_k)$, with $1\le \xi_i\le p^{c+1}$ and $\xi_i\equiv \xi\pmod{p^c}$ $(1\le i\le k)$, and satisfying the property that $\xi_i\equiv \xi_j\pmod{p^{c+1}}$ for no $i$ and $j$ with $1\le i<j\le k$. In addition, write $\Sig_k=\{1,-1\}^k$, and consider an element $\bfsig$ of $\Sig_k$. Recalling the definition (\ref{2.2}), we then put
\begin{equation}\label{3.3}
\grF_c^\bfsig (\bfalp ;\xi)=\sum_{\bfxi \in \Xi_c(\xi)}\prod_{i=1}^k\grf_{c+1}(\sig_i\bfalp;\xi_i).
\end{equation}

\par Two mixed mean values play leading roles in our arguments. When $a$ and $b$ are non-negative integers, and $\bfsig, \bftau\in \Sig_k$, we define
\begin{equation}\label{3.4}
I^\bfsig_{a,b}(X;\xi,\eta)=\oint |\grF_a^\bfsig (\bfalp;\xi)^2\grf_b(\bfalp;\eta)^{2s}|\d\bfalp
\end{equation}
and
\begin{equation}\label{3.5}
K^{\bfsig,\bftau}_{a,b}(X;\xi,\eta)=\oint |\grF_a^\bfsig(\bfalp;\xi)^2\grF_b^\bftau (\bfalp;\eta)^{2u}|\d \bfalp .
\end{equation}
It is convenient then to put
\begin{equation}\label{3.6}
I_{a,b}(X)=\max_{1\le \xi\le p^a}\max_{1\le \eta\le p^b}\max_{\bfsig \in \Sig_k}I^\bfsig_{a,b}(X;\xi,\eta)
\end{equation}
and
\begin{equation}\label{3.7}
K_{a,b}(X)=\max_{1\le \xi\le p^a}\max_{1\le \eta\le p^b}\max_{\bfsig,\bftau\in \Sig_k}K_{a,b}^{\bfsig,\bftau}(X;\xi,\eta).
\end{equation}
Notice here that these mean values depend on our choice of $p$. However, since we will shortly fix this choice of $p$ once and for all, we suppress mention of this prime when referring to $I_{a,b}(X)$ and $K_{a,b}(X)$.\par

Our arguments are simplified considerably by making transparent the relationship between various mean values on the one hand, and the anticipated magnitude of these mean values on the other. Of course, such a concept may not be well-defined, and so we indicate in what follows quite concretely what is intended. We define the {\it normalised magnitude} of a mean value $\calM$ relative to its anticipated size $\calM^*$ to be $\calM/\calM^*$, a quantity we denote by $\ldbrack \calM \rdbrack$. In particular, we define
\begin{equation}\label{3.8}
\ldbrack J_t(X)\rdbrack =\frac{J_{t,k}(X)}{X^{2t-\frac{1}{2}k(k+1)}}\quad (t=s,s+k),
\end{equation}
and when $0\le a<b$, we define
\begin{align}
\ldbrack I_{a,b}(X)\rdbrack &=\frac{I_{a,b}(X)}{(X/M^b)^{2s}(X/M^a)^{2k-\frac{1}{2}k(k+1)}},\notag \\
\ldbrack K_{a,b}(X)\rdbrack &=\frac{K_{a,b}(X)}{(X/M^b)^{2s}(X/M^a)^{2k-\frac{1}{2}k(k+1)}}.\label{3.9}
\end{align}
Note that the lower bound (\ref{3.1}) implies that
\begin{equation}\label{3.10}
\ldbrack J_{s+k}(X)\rdbrack >X^{\eta_{s+k}-\del},
\end{equation}
while the upper bound (\ref{3.2}) ensures that, whenever $X^{\del^2}<Y\le X$, one has
\begin{equation}\label{3.11}
\ldbrack J_t(Y)\rdbrack <Y^{\eta_t+\del}\quad (t=s,s+k).
\end{equation}

\par Mean values of the exponential sum $\grf_c(\bfalp;\xi)$ are easily bounded by exploiting the translation-dilation invariance of the solution sets of the system of equations (\ref{1.3}). The argument is relatively familiar, though we provide details for the sake of completeness.

\begin{lemma}\label{lemma3.1}
Suppose that $c$ is a non-negative integer with $c\tet\le 1$. Then for each natural number $t$, one has
\begin{equation}\label{3.12a}
\max_{1\le \xi\le p^c}\oint |\grf_c(\bfalp;\xi)|^{2t}\,d\bfalp \ll_t J_t(X/M^c).
\end{equation}
\end{lemma}

\begin{proof} Let $\xi$ be an integer with $1\le \xi\le p^c$. From the definition (\ref{2.2}) of the exponential sum $\grf_c(\bfalp;\xi)$, one has
$$\grf_c(\bfalp;\xi)=\sum_{(1-\xi)/p^c\le y\le (X-\xi)/p^c}e(\psi (p^cy+\xi;\bfalp)),$$
in which $\psi(z;\bfalp)$ is given by (\ref{2.3}). By orthogonality, therefore, one finds that the integral on the left hand side of (\ref{3.12a}) counts the number of integral solutions of the system of equations
\begin{equation}\label{3.12}
\sum_{i=1}^t(p^cy_i+\xi)^j=\sum_{i=1}^t(p^cz_i+\xi)^j\quad (1\le j\le k),
\end{equation}
with $0\le \bfy, \bfz\le (X-\xi)/p^c$. An application of the Binomial Theorem shows that the pair $\bfy,\bfz$ satisfies (\ref{3.12}) if and only if it satisfies the system
$$\sum_{i=1}^ty_i^j=\sum_{i=1}^tz_i^j\quad (1\le j\le k).$$
Thus, on considering the underlying Diophantine system and recalling (\ref{1.1}) and (\ref{1.2}), we find that
\begin{align*}
\oint |\grf_c(\bfalp;\xi)|^{2t}\d \bfalp &\le \oint |1+f(\bfalp ;X/p^c)|^{2t}\d \bfalp \\
&\ll_t 1+\oint |f(\bfalp ;X/p^c)|^{2t}\d \bfalp \\
&=1+J_t(X/p^c).
\end{align*}
The desired conclusion follows on noting that diagonal solutions alone ensure that $J_t(X/M^c)\ge 1$.
\end{proof}

Our next preparatory manoeuvre concerns the initiation of the iterative procedure, and it is here that we fix our choice for $p$. It is convenient here and elsewhere to write ${\mathbf 1}$ for the $k$-tuple $(1,\ldots,1)$.

\begin{lemma}\label{lemma3.2}
There exists a prime number $p$ with $M<p\le 2M$ for which $J_{s+k}(X)\ll M^{2s}I_{0,1}(X)$.
\end{lemma}

\begin{proof} The quantity $J_{s+k}(X)$ counts the number of integral solutions of the system
$$\sum_{i=1}^{s+k}(x_i^j-y_i^j)=0\quad (1\le j\le k),$$
with $1\le \bfx,\bfy\le X$. Let $T_0$ denote the number of such solutions in which $x_i=x_j$ for some $i$ and $j$ with $1\le i<j\le k$, and let $T_1$ denote the corresponding number of solutions with $x_i=x_j$ for no $i$ and $j$ with $1\le i<j\le k$.\par

On considering the underlying Diophantine system, one finds that
$$T_0\ll \oint f(2\bfalp ;X)f(\bfalp ;X)^{s+k-2}f(-\bfalp ;X)^{s+k}\d\bfalp ,$$
whence by H\"older's inequality, it follows that
$$T_0\ll \Bigl( \oint |f(\bfalp ;X)|^{2s+2k}\d\bfalp \Bigr)^{1-1/(s+k)}\Bigl( \oint |f(2\bfalp ;X)|^{2s+2k}\d\bfalp \Bigr)^{1/(2s+2k)}.$$
Thus, by a change of variables, we obtain the upper bound
\begin{equation}\label{3.13}
T_0\ll (J_{s+k}(X))^{1-1/(2s+2k)}.
\end{equation}

\par Consider next a solution $\bfx,\bfy$ counted by $T_1$. Write
$$\Del(\bfx)=\prod_{1\le i<j\le k}|x_i-x_j|,$$
and note that $0<\Del (\bfx) <X^{k(k-1)}$. Let $\calP$ denote any set of $[k^3/\tet]+1$ distinct prime numbers with $M<p\le 2M$. Such a set exists by the Prime Number Theorem. It follows that
$$\prod_{p\in \calP}p>M^{k^3/\tet}=X^{k^3}>\Del(\bfx),$$
and hence one at least of the elements of $\calP$ does not divide $\Del(\bfx)$. In particular, there exists a prime $p\in \calP$ for which $x_i\equiv x_j\pmod{p}$ for no $i$ and $j$ with $1\le i<j\le k$. On considering the underlying Diophantine system, we therefore see that
$$T_1\ll \sum_{p\in \calP}\oint \grF_0^{\mathbf 1}(\bfalp ;0)f(\bfalp;X)^sf(-\bfalp;X)^{s+k}\d\bfalp .$$
Therefore, as a consequence of Schwarz's inequality, one finds that
\begin{align*}
T_1&\ll \max_{p\in \calP}\Bigl( \oint |\grF_0^{\mathbf 1}(\bfalp;0)^2f(\bfalp ;X)^{2s}|\d\bfalp \Bigr)^{1/2}\Bigl( \oint |f(\bfalp;X)|^{2s+2k}\d\bfalp\Bigr)^{1/2}\\
&=\max_{p\in \calP}\Bigl( \oint |\grF_0^{\mathbf 1}(\bfalp ;0)^2\grf_0(\bfalp ;0)^{2s}|\d\bfalp \Bigr)^{1/2}\bigl( J_{s+k}(X)\Bigr)^{1/2}.
\end{align*}
In this way, we deduce that a prime number $p$ with $M<p\le 2M$ exists for which
\begin{equation}\label{3.14}
T_1\ll (I_{0,0}(X))^{1/2}(J_{s+k}(X))^{1/2}.
\end{equation}
On recalling that $J_{s+k}(X)=T_0+T_1$, we find from (\ref{3.13}) and (\ref{3.14}) that
\begin{equation}\label{3.15}
J_{s+k}(X)\ll 1+I_{0,0}(X)\ll I_{0,0}(X).
\end{equation}

\par Next, we split the summation in the definition (\ref{2.2}) of $\grf_0(\bfalp;0)$ into arithmetic progressions modulo $p$. Thus we obtain
$$\grf_0(\bfalp;0)=\sum_{\xi=1}^p\grf_1(\bfalp;\xi),$$
whence by H\"older's inequality one has
$$|\grf_0(\bfalp;0)|^{2s}\le p^{2s-1}\sum_{\xi=1}^p|\grf_1(\bfalp;\xi)|^{2s}.$$
It therefore follows from (\ref{3.4}) and (\ref{3.6}) that
\begin{equation}\label{3.16}
I_{0,0}(X)\ll M^{2s}\max_{1\le \xi\le p}\max_{\bfsig \in \Sig_k}\oint |\grF_0^\bfsig (\bfalp;0)^2\grf_1(\bfalp;\xi)^{2s}|\d\bfalp \le M^{2s}I_{0,1}(X).
\end{equation}
The conclusion of the lemma is obtained by subtituting (\ref{3.16}) into (\ref{3.15}).
\end{proof}

We now fix the prime number $p$, once and for all, so that the upper bound $J_{s+k}(X)\ll M^{2s}I_{0,1}(X)$ holds.

\section{The auxiliary system of congruences} The efficient congruencing process delivers a strong congruence condition on a subset of variables. In order to be useful in further congruencing activities, this condition must be converted into a restriction of certain variables to higher level arithmetic progressions. It is to this task that we attend in the present section.\par

When $\bfsig \in \Sig_k$, denote by $\calB^\bfsig_{a,b}(\bfm;\xi,\eta)$ the set of solutions of the system of congruences
\begin{equation}\label{4.1}
\sum_{i=1}^k\sig_i(z_i-\eta)^j\equiv m_j\pmod{p^{jb}}\quad (1\le j\le k),
\end{equation}
with $1\le z_i\le p^{kb}$ and $\bfz\equiv \bfxi\pmod{p^{a+1}}$ for some $\bfxi\in \Xi_a(\xi)$.

\begin{lemma}\label{lemma4.1} Suppose that $a$ and $b$ are non-negative integers with $b>a$. Then
$$\max_{1\le \xi\le p^a}\max_{1\le \eta\le p^b}\max_{\bfsig\in \Sig_k}\text{card}\left( \calB^\bfsig_{a,b}(\bfm;\xi,\eta)\right) \le k!p^{\frac{1}{2}k(k-1)(a+b)}.$$
\end{lemma}

\begin{proof} Consider fixed integers $a$ and $b$ with $0\le a<b$, a fixed $k$-tuple $\bfsig\in \Sig_k$, and fixed integers $\xi$ and $\eta$ with $1\le \xi\le p^a$ and $1\le \eta\le p^b$. Denote by $\calD_1(\bfn)$ the set of solutions of the system of congruences
\begin{equation}\label{4.2}
\sum_{i=1}^k\sig_i(z_i-\eta)^j\equiv n_j\pmod{p^{kb}}\quad (1\le j\le k),
\end{equation}
with $1\le \bfz\le p^{kb}$ and $\bfz\equiv \bfxi\pmod{p^{a+1}}$ for some $\bfxi\in \Xi_a(\xi)$. Then it follows from (\ref{4.1}) that we have
$$\text{card}(\calB^\bfsig_{a,b}(\bfm;\xi,\eta))=\sum_{\substack{n_1\equiv m_1\mmod{p^b}\\ 1\le n_1\le p^{kb}}}\ldots \sum_{\substack{n_k\equiv m_k\mmod{p^{kb}}\\ 1\le n_k\le p^{kb}}}\text{card}(\calD_1(\bfn)).$$
Counting the number of $k$-tuples $\bfn$ with $1\le \bfn\le p^{kb}$ for which $n_j\equiv m_j\pmod{p^{jb}}$ $(1\le j\le k)$, therefore, we see that
\begin{equation}\label{4.3}
\text{card}(\calB^\bfsig_{a,b}(\bfm;\xi,\eta))\le p^{\frac{1}{2}k(k-1)b}\max_{1\le \bfn\le p^{kb}}\text{card}(\calD_1(\bfn)).
\end{equation}

\par We now examine the system (\ref{4.2}). We begin by rewriting each variable $z_i$ in the shape $z_i=p^ay_i+\xi$. In view of the hypothesis that $\bfz\equiv \bfxi\pmod{p^{a+1}}$ for some $\bfxi\in \Xi_a(\xi)$, we find that the $k$-tuple $\bfy$ satisfies the condition that $y_i\equiv y_j\pmod{p}$ for no $i$ and $j$ with $1\le i<j\le k$. With this substitution in (\ref{4.2}), we find by the Binomial Theorem that the set of solutions $\calD_1(\bfn)$ is in bijective correspondence with the set of solutions of the system of congruences
\begin{equation}\label{4.4}
\sum_{l=0}^j\binom{j}{l}(\xi-\eta)^{j-l}p^{la}\sum_{i=1}^k\sig_iy_i^l\equiv m_j\pmod{p^{kb}}\quad (1\le j\le k),
\end{equation}
with $1\le \bfy\le p^{kb-a}$. Let $\bfy=\bfw$ be any solution of this system, if indeed a solution exists. Then it follows from (\ref{4.4}) that all other solutions $\bfy$ satisfy the system of congruences
\begin{equation}\label{4.5}
\sum_{l=0}^j\binom{j}{l}(\xi-\eta)^{j-l}p^{la}\sum_{i=1}^k\sig_i(y_i^l-w_i^l)\equiv 0\pmod{p^{kb}}\quad (1\le j\le k).
\end{equation}
By taking linear combinations of the congruences here, we find that the system (\ref{4.5}) is equivalent to the new system
$$\sum_{i=1}^k\sig_iy_i^j\equiv \sum_{i=1}^k\sig_iw_i^j\pmod{p^{kb-ja}}\quad (1\le j\le k).$$

Next, we write $\calD_2(\bfu)$ for the set of solutions of the system of congruences
$$\sum_{i=1}^k\sig_iy_i^j\equiv u_j\pmod{p^{kb-ja}}\quad (1\le j\le k),$$
with $1\le \bfy\le p^{kb-a}$ and $y_i\equiv y_j\pmod{p}$ for no $i$ and $j$ with $1\le i<j\le k$. Then it follows from our discussion thus far that
\begin{equation}\label{4.7}
\text{card}(\calD_1(\bfn))\le \max_{1\le \bfu\le p^{kb-a}}\text{card}(\calD_2(\bfu)).
\end{equation}
Denote by $\calD_3(\bfv)$ the set of solutions of the system of congruences
\begin{equation}\label{4.8}
\sum_{i=1}^k\sig_iy_i^j\equiv v_j\pmod{p^{kb-a}}\quad (1\le j\le k),
\end{equation}
with $1\le \bfy\le p^{kb-a}$ and $y_i\equiv y_j\pmod{p}$ for no $i$ and $j$ with $1\le i<j\le k$. Then we have
$$\text{card}(\calD_2(\bfu))\le \sum_{\substack{v_1\equiv u_1\mmod{p^{kb-a}}\\ 1\le v_1\le p^{kb-a}}}\ldots \sum_{\substack{v_k\equiv u_k\mmod{p^{kb-ka}}\\ 1\le v_k\le p^{kb-a}}}\text{card}(\calD_3(\bfv)).$$
Counting the number of $k$-tuples $\bfv$ with $1\le \bfv\le p^{kb-a}$ for which $v_j\equiv u_j\pmod{p^{kb-ja}}$ $(1\le j\le k)$, therefore, we deduce that
$$\text{card}(\calD_2(\bfu))\le p^{\frac{1}{2}k(k-1)a}\max_{1\le \bfv\le p^{kb-a}}\text{card}(\calD_3(\bfv)).$$
Consequently, in combination with (\ref{4.3}) and (\ref{4.7}), we have shown thus far that
\begin{equation}\label{4.9}
\text{card}(\calB^\bfsig_{a,b}(\bfm;\xi,\eta))\le p^{\frac{1}{2}k(k-1)(a+b)}\max_{1\le \bfv\le p^{kb-a}}\text{card}(\calD_3(\bfv)).
\end{equation}

\par Suppose now that $\bfy=\bfz$ is any solution of (\ref{4.8}) belonging to $\calD_3(\bfv)$, if one exists. Then all other solutions $\bfy$ satisfy the system
$$\sum_{i=1}^k\sig_iy_i^j\equiv \sum_{i=1}^k\sig_iz_i^j\pmod{p^{kb-a}}\quad (1\le j\le k).$$
Let $\calI$ denote the set of indices $i$ with $1\le i\le k$ for which $\sig_i=1$, and let $\calJ$ denote the corresponding set of indices for which $\sig_i=-1$. Then this system of congruences is equivalent to the new system
$$\sum_{i\in \calI}y_i^j+\sum_{l\in \calJ}z_l^j\equiv \sum_{i\in \calI}z_i^j+\sum_{l\in \calJ}y_l^j\pmod{p^{kb-a}}\quad (1\le j\le k).$$
We are at liberty to assume that $p>k$. Consequently, from Newton's formulae relating the sums of powers of the roots of a polynomial with its coefficients, we find that
$$\prod_{i\in \calI}(t-y_i)\prod_{l\in \calJ}(t-z_l)\equiv \prod_{j\in \calI}(t-z_j)\prod_{m\in \calJ}(t-y_m)\pmod{p^{kb-a}}.$$
But $z_l\equiv z_m\pmod{p}$ for no $l$ and $m$ with $1\le l<m\le k$. Then for each $j$ with $j\in \calI$, by putting $t=z_j$ we deduce that
$$\prod_{i\in \calI}(z_j-y_i)\prod_{l\in \calJ}(z_j-z_l)\equiv 0\pmod{p^{kb-a}},$$
whence for some $i$ with $i\in \calI$ one has $y_i\equiv z_j\pmod{p^{kb-a}}$. Similarly, for each $l$ with $l\in \calJ$, we deduce that for some $m$ with $m\in \calJ$, one has $y_m\equiv z_l\pmod{p^{kb-a}}$. It follows that the sets $\{y_1,\dots ,y_k\}$ and $\{z_1,\dots ,z_k\}$ are mutually congruent modulo $p^{kb-a}$, whence $\text{card}(\calD_3(\bfv))\le k!$. The conclusion of the lemma now follows at once from (\ref{4.9}).
\end{proof}

\section{The conditioning process} The mean value $I^\bfsig_{a,b}(X)$, defined via (\ref{3.4}), is already in a form suitable for the extraction of an efficient congruence. Unfortunately, however, one would be poorly positioned to extract the next efficient congruence following the one at hand were one not to plan ahead by conditioning the auxiliary variables encoded by the exponential sum $\grf_b(\bfalp;\eta)$. In this section we show that the factor $\grf_b(\bfalp;\eta)^{2s}$ occurring in (\ref{3.4}) can, in essence, be replaced by the conditioned factor $\grF_b^\bftau (\bfalp;\eta)^{2u}$. The latter involves $k$-tuples of variables in residue classes distinct modulo $p^{b+1}$, and is suitable for subsequent congruencing operations.

\begin{lemma}\label{lemma5.1}
Let $a$ and $b$ be integers with $b>a\ge 0$. Then one has
$$I_{a,b}(X)\ll K_{a,b}(X)+M^{k-1}I_{a,b+1}(X).$$
\end{lemma}

\begin{proof} Consider fixed integers $\xi$ and $\eta$ with $1\le \xi\le p^a$ and $1\le \eta\le p^b$, and a $k$-tuple $\bfsig\in\Sig_k$. Then on considering the underlying Diophantine system, one finds from (\ref{3.4}) that $I_{a,b}^\bfsig(X;\xi,\eta)$ counts the number of integral solutions of the system
\begin{equation}\label{5.1}
\sum_{i=1}^k\sig_i(x_i^j-y_i^j)=\sum_{l=1}^s(v_l^j-w_l^j)\quad (1\le j\le k),
\end{equation}
with
$$1\le \bfx,\bfy,\bfv,\bfw \le X,\quad \bfv\equiv \bfw\equiv \eta\pmod{p^b},$$
and for some $\bfxi,\bfzet \in \Xi_a(\xi)$, with $\bfx,\bfy$ subject to the additional condition
$$\bfx\equiv \bfxi\pmod{p^{a+1}}\quad \text{and}\quad \bfy\equiv \bfzet\pmod{p^{a+1}}.$$
Let $T_1$ denote the number of integral solutions $\bfx,\bfy,\bfv,\bfw$ of the system (\ref{5.1}), counted by $I^\bfsig_{a,b}(X;\xi,\eta)$, in which $v_1,\ldots ,v_s$ and $w_1,\ldots ,w_s$ together lie in at most $k-1$ distinct residue classes modulo $p^{b+1}$, and let $T_2$ denote the corresponding number of solutions in which the integers $v_1,\ldots ,v_s$ and $w_1,\ldots ,w_s$ together contain at least $k$ distinct residue classes modulo $p^{b+1}$. Then we have
$$I_{a,b}^\bfsig(X;\xi,\eta)\le T_1+T_2.$$

\par On considering the underlying Diophantine system, it is apparent that
$$T_1\ll \sum_{\substack{1\le \eta_1,\ldots ,\eta_{k-1}\le p^{b+1}\\ \bfeta\equiv \eta\mmod{p^b}}}\sum_{0\le \bfe\le 2s}\oint |\grF_a^\bfsig(\bfalp;\xi)^2\grf_{b+1}(\bfalp;\eta_1)^{e_1}\ldots \grf_{b+1}(\bfalp;\eta_{k-1})^{e_{k-1}}|\d\bfalp ,$$
in which the summation over $\bfe$ is subject to the condition
$$e_1+e_2+\dots +e_{k-1}=2s.$$
In view of the elementary inequality
$$|z_1\dots z_n|\le |z_1|^n+\dots +|z_n|^n,$$
we find that
$$|\grf_{b+1}(\bfalp;\eta_1)^{e_1}\dots \grf_{b+1}(\bfalp;\eta_{k-1})^{e_{k-1}}|\le \sum_{i=1}^{k-1}|\grf_{b+1}(\bfalp;\eta_i)|^{2s}.$$
Thus we deduce that
\begin{align}
T_1&\ll \sum_{\substack{1\le \eta_1,\ldots ,\eta_{k-1}\le p^{b+1}\\ \bfeta\equiv \eta\mmod{p^b}}}\sum_{i=1}^{k-1}\oint |\grF_a^\bfsig(\bfalp;\xi)^2\grf_{b+1}(\bfalp;\eta_i)^{2s}|\d\bfalp\notag \\
&\ll p^{k-1}\max_{1\le \eta_0\le p^{b+1}}I_{a,b+1}^\bfsig(X;\xi,\eta_0).\label{5.2}
\end{align}

\par We turn our attention next to the solutions $\bfx,\bfy,\bfv,\bfw$ counted by $T_2$. The integers $v_1,\dots ,v_s$ and $w_1,\dots ,w_s$ now lie together in $k$ at least distinct residue classes modulo $p^{b+1}$. By relabelling variables if necessary, therefore, there is no loss of generality in supposing that $v_1,\dots ,v_k$ lie in distinct residue classes modulo $p^{b+1}$. On considering the underlying Diophantine system, we thus deduce that for some $\bftau\in\Sig_k$, one has
$$T_2\ll \oint |\grF_a^\bfsig(\bfalp;\xi)|^2\grF_b^\bftau(\bfalp;\eta)\grf_b(\bfalp;\eta)^{s-r_+}\grf_b(-\bfalp;\eta)^{s-r_-}\d\bfalp .$$
Here, we have written $r_+$ for the number of the coordinates of $\bftau$ which are $+1$, and $r_-$ for the number which are $-1$. Thus, in particular, one has $r_++r_-=k$. On recalling that $s=uk$, an application of H\"older's inequality leads from here to the bound
$$T_2\ll \Bigl( \oint |\grF_a^\bfsig(\bfalp;\xi)^2\grF_b^\bftau(\bfalp;\eta)^{2u}|\d\bfalp \Bigr)^{1/(2u)}\Bigl( \oint |\grF_a^\bfsig(\bfalp;\xi)^2\grf_b(\bfalp;\eta)^{2s}|\d\bfalp \Bigr)^{1-1/(2u)}.$$
Hence, in view of the definitions (\ref{3.4}) and (\ref{3.5}), we arrive at the estimate
\begin{equation}\label{5.3}
T_2\ll (K_{a,b}^{\bfsig,\bftau}(X;\xi,\eta))^{1/(2u)}(I_{a,b}^\bfsig(X;\xi,\eta))^{1-1/(2u)}.
\end{equation}

\par Combining (\ref{5.2}) and (\ref{5.3}), and recalling (\ref{3.6}) and (\ref{3.7}), we deduce that
$$I_{a,b}(X)\ll M^{k-1}I_{a,b+1}(X)+(K_{a,b}(X))^{1/(2u)}(I_{a,b}(X))^{1-1/(2u)}.$$
The conclusion of the lemma now follows on disentangling this inequality.
\end{proof}
 
Repeated application of Lemma \ref{lemma5.1} shows that whenever $a$, $b$ and $H$ are non-negative integers with $b>a\ge 0$, then
\begin{equation}\label{5.4}
I_{a,b}(X)\ll \sum_{h=0}^{H-1}M^{h(k-1)}K_{a,b+h}(X)+M^{H(k-1)}I_{a,b+H}(X).
\end{equation}
Since for large values of $H$, quantities of the type $I_{a,b+H}(X)$ are an irritant to our argument, we show in the next lemma that values of $H$ exceeding $\frac{1}{2}(b-a)$ are harmless.

\begin{lemma}\label{lemma5.2}
Let $a$, $b$ and $H$ be non-negative integers with
$$0<{\textstyle{\frac{1}{2}}}(b-a)\le H\le \tet^{-1}-b.$$
Then one has
$$M^{H(k-1)}I_{a,b+H}(X)\ll M^{-H/2}(X/M^b)^{2s}(X/M^a)^{2k-\frac{1}{2}k(k+1)+\eta_{s+k}}.$$
\end{lemma}

\begin{proof} On considering the underlying Diophantine systems, it follows from (\ref{3.3}) and (\ref{3.4}) that when $1\le \xi\le p^a$ and $1\le \eta\le p^{b+H}$, and $\bfsig\in\Sig_k$, one has
$$I^\bfsig_{a,b+H}(X;\xi,\eta)\le \oint |\grf_a(\bfalp;\xi)^{2k}\grf_{b+H}(\bfalp;\eta)^{2s}|\d\bfalp .$$
Then an application of H\"older's inequality in combination with Lemma \ref{lemma3.1} leads to the upper bound
\begin{align*}
I^\bfsig_{a,b+H}(X;\xi,\eta)&\le \Bigl( \oint |\grf_a(\bfalp;\xi)|^{2s+2k}\d\bfalp \Bigr)^{k/(s+k)} \Bigl( \oint |\grf_{b+H}(\bfalp;\eta)|^{2s+2k}\d\bfalp \Bigr)^{s/(s+k)}\\
&\ll (J_{s+k}(X/M^a))^{k/(s+k)}(J_{s+k}(X/M^{b+H}))^{s/(s+k)}.
\end{align*}
Consequently, in view of (\ref{3.2}), we have
\begin{align}
I_{a,b+H}(X)&\ll ((X/M^a)^{k/(s+k)}(X/M^{b+H})^{s/(s+k)})^{2s+2k-\frac{1}{2}k(k+1)+\eta_{s+k}+\del}\notag \\
&\ll X^\del (X/M^a)^{2k-\frac{1}{2}k(k+1)+\eta_{s+k}}(X/M^b)^{2s}\Ups,\label{5.5}
\end{align}
where
$$\Ups=(M^{b-a+H})^{\frac{1}{2}k(k+1)s/(s+k)}M^{-2sH}.$$
But when $s\ge k^2$ and $H\ge \frac{1}{2}(b-a)$, one has
\begin{align*}
\frac{s}{s+k}\left(2(s+k)-{\textstyle \frac{1}{2}}k(k+1)\right)H&\ge \frac{s}{s+k}\left({\textstyle \frac{3}{2}}k(k+1)\right)H\\
&\ge \frac{s}{s+k}\left({\textstyle\frac{1}{2}}k(k+1)\right)(b-a)+{\textstyle\frac{1}{2}}k^2H.
\end{align*}
Thus we see that for $k\ge 2$, one has
$$M^{H(k-1)}\Ups\le M^{H(k-1-\frac{1}{2}k^2)}\le M^{-H},$$
whence
$$X^\del M^{H(k-1)}\Ups \le M^{-H/2}.$$
The conclusion of the lemma follows on substituting this estimate into (\ref{5.5}).
\end{proof}

Combining Lemma \ref{lemma5.2} with the upper bound (\ref{5.4}), we may conclude as follows. Here, as usual, when $\bet\in \dbR$ we write $\lceil \bet\rceil $ for the least integer no smaller than $\bet$.

\begin{lemma}\label{lemma5.3}
Let $a$ and $b$ be integers with $0\le a<b$, and put $H=\lceil \frac{1}{2}(b-a)\rceil$. Suppose that $b+H\le \tet^{-1}$. Then there exists an integer $h$ with $0\le h<H$ having the property that
$$I_{a,b}(X)\ll M^{h(k-1)}K_{a,b+h}(X)+M^{-H/2}(X/M^b)^{2s}(X/M^a)^{2k-\frac{1}{2}k(k+1)+\eta_{s+k}}.$$
\end{lemma}

By making use of the special case of Lemma \ref{lemma5.3} in which $a=0$ and $b=1$, we are able to refine Lemma \ref{lemma3.2} into a form more directly applicable.

\begin{lemma}\label{lemma5.4}
One has $J_{s+k}(X)\ll M^{2s}K_{0,1}(X)$.
\end{lemma}

\begin{proof} Observe first that when $a=0$ and $b=1$, then $\lceil \frac{1}{2}(b-a)\rceil =1$. Thus we deduce from Lemma \ref{lemma5.3} that
$$I_{0,1}(X)\ll K_{0,1}(X)+M^{-1/2}(X/M)^{2s}X^{2k-\frac{1}{2}k(k+1)+\eta_{s+k}}.$$
Since we may suppose that $M^{1/2}>X^{4\del}$, it follows from Lemma \ref{lemma3.2} that
$$J_{s+k}(X)\ll M^{2s}I_{0,1}(X)\ll M^{2s}K_{0,1}(X)+X^{2s+2k-\frac{1}{2}k(k+1)+\eta_{s+k}-2\del}.$$
But in view of (\ref{3.10}), we have
$$J_{s+k}(X)\gg X^{2s+2k-\frac{1}{2}k(k+1)+\eta_{s+k}-\del},$$
and hence we arrive at the upper bound
$$J_{s+k}(X)\ll M^{2s}K_{0,1}(X)+X^{-\del}J_{s+k}(X).$$
The conclusion of the lemma follows on disentangling this inequality.
\end{proof}

\section{The efficient congruencing step} The mean value $K_{a,b}(X)$ contains a powerful latent congruence condition. Our task in this section is to convert this condition into one that may be exploited by means of an iterative procedure.

\begin{lemma}\label{lemma6.1}
Suppose that $a$ and $b$ are integers with $0\le a<b\le \tet^{-1}$. Then one has
$$K_{a,b}(X)\ll M^{\frac{1}{2}k(k-1)(b+a)}(M^{kb-a})^k\left( J_{s+k}(X/M^b)\right)^{1-k/s}\left( I_{b,kb}(X)\right)^{k/s}.$$
\end{lemma}

\begin{proof} Consider fixed integers $\xi$ and $\eta$ with $1\le \xi\le p^a$ and $1\le \eta\le p^b$, and $k$-tuples $\bfsig,\bftau\in\Sig_k$. Then on considering the underlying Diophantine system, one finds from (\ref{3.5}) that $K_{a,b}^{\bfsig,\bftau}(X;\xi,\eta)$ counts the number of integral solutions of the system
\begin{equation}\label{6.3}
\sum_{i=1}^k\sig_i(x_i^j-y_i^j)=\sum_{l=1}^u\sum_{m=1}^k\tau_m(v_{lm}^j-w_{lm}^j)\quad (1\le j\le k),
\end{equation}
in which, for some $\bfxi,\bfzet \in \Xi_a(\xi)$, one has
$$1\le \bfx,\bfy\le X,\quad \bfx\equiv \bfxi\pmod{p^{a+1}}\quad\text{and}\quad  \bfy\equiv \bfzet\pmod{p^{a+1}},$$
and for $1\le l\le u$, for some $\bfeta_l,\bfnu_l\in \Xi_b(\eta)$, one has
$$1\le \bfv_l,\bfw_l\le X,\quad \bfv_l\equiv \bfeta_l\pmod{p^{b+1}}\quad \text{and}\quad \bfw_l\equiv \bfnu_l\pmod{p^{b+1}}.$$
By applying the Binomial Theorem, we see that the system (\ref{6.3}) is equivalent to the new system of equations
\begin{equation}\label{6.4}
\sum_{i=1}^k\sig_i((x_i-\eta)^j-(y_i-\eta)^j)=\sum_{l=1}^u\sum_{m=1}^k\tau_m((v_{lm}-\eta)^j-(w_{lm}-\eta)^j)\quad (1\le j\le k).
\end{equation}
But in any solution $\bfx,\bfy,\bfv,\bfw$ counted by $K^{\bfsig,\bftau}_{a,b}(X;\xi,\eta)$, one has $\bfv\equiv \bfw\equiv \eta\pmod{p^b}$. We therefore deduce from (\ref{6.4}) that
\begin{equation}\label{6.5}
\sum_{i=1}^k\sig_i(x_i-\eta)^j\equiv \sum_{i=1}^k\sig_i(y_i-\eta)^j\pmod{p^{jb}}\quad (1\le j\le k).
\end{equation}

\par Recall the notation from the preamble to Lemma \ref{lemma4.1}, and write
$$\grG_{a,b}^\bfsig(\bfalp;\xi,\eta;\bfm)=\sum_{\bfzet \in \calB^\bfsig_{a,b}(\bfm;\xi,\eta)}\prod_{i=1}^k\grf_{kb}(\sig_i\bfalp;\zet_i).$$
Then on considering the underlying Diophantine system, it follows from (\ref{6.3}) and (\ref{6.5}) that
\begin{equation}\label{6.6}
K_{a,b}^{\bfsig,\bftau}(X;\xi,\eta)=\sum_{m_1=1}^{p^b}\dots \sum_{m_k=1}^{p^{kb}}\oint |\grG_{a,b}^\bfsig(\bfalp;\xi,\eta;\bfm)^2\grF_b^\bftau(\bfalp;\eta)^{2u}|\d\bfalp .
\end{equation}
An application of Cauchy's inequality in combination with Lemma \ref{lemma4.1} yields the upper bound
\begin{align}
|\grG_{a,b}^\bfsig(\bfalp;\xi,\eta;\bfm)|^2&\le \text{card}(\calB_{a,b}^\bfsig(\bfm;\xi,\eta))\sum_{\bfzet\in \calB^\bfsig_{a,b}(\bfm;\xi,\eta)}\prod_{i=1}^k|\grf_{kb}(\bfalp;\zet_i)|^2\notag \\
&\ll M^{\frac{1}{2}k(k-1)(a+b)}\sum_{\bfzet \in \calB^\bfsig_{a,b}(\bfm;\xi,\eta)}\prod_{i=1}^k|\grf_{kb}(\bfalp;\zet_i)|^2.\label{6.7}
\end{align}
Next, on substituting (\ref{6.7}) into (\ref{6.6}) and considering the underlying Diophantine system, we deduce that
\begin{equation}\label{6.8}
K_{a,b}^{\bfsig,\bftau}(X;\xi,\eta)\ll M^{\frac{1}{2}k(k-1)(a+b)}\sum_{\substack{1\le \bfzet\le p^{kb}\\ \bfzet \equiv \xi\mmod{p^a}}}\oint \Bigl( \prod_{i=1}^k|\grf_{kb}(\bfalp;\zet_i)|^2\Bigr) |\grF_b^\bftau(\bfalp;\eta)|^{2u}\d\bfalp .
\end{equation}

\par Observe next that by H\"older's inequality, one has
\begin{align*}
\sum_{\substack{1\le \bfzet\le p^{kb}\\ \bfzet \equiv \xi\mmod{p^a}}}\prod_{i=1}^k|\grf_{kb}(\bfalp;\zet_i)|^2&=\Bigl( \sum_{\substack{1\le \zet\le p^{kb}\\ \zet\equiv \xi\mmod{p^a}}}|\grf_{kb}(\bfalp;\zet)|^2\Bigr)^k\\
&\le (p^{kb-a})^{k-1}\sum_{\substack{1\le \zet\le p^{kb}\\ \zet\equiv \xi\mmod{p^a}}}|\grf_{kb}(\bfalp;\zet)|^{2k}.
\end{align*}
Then it follows from (\ref{6.8}) that
\begin{equation}\label{6.9}
K_{a,b}^{\bfsig,\bftau}(X;\xi,\eta)\ll M^{\frac{1}{2}k(k-1)(a+b)}(M^{kb-a})^k\max_{1\le \zet\le p^{kb}}\oint |\grf_{kb}(\bfalp;\zet)^{2k}\grF_b^\bftau(\bfalp;\eta)^{2u}|\d\bfalp .
\end{equation}
On recalling that $s=uk$, an application of H\"older's inequality supplies the bound
\begin{equation}\label{6.9c}
\oint |\grf_{kb}(\bfalp;\zet)^{2k}\grF_b^\bftau(\bfalp;\eta)^{2u}|\d\bfalp \le U_1^{1-k/s}U_2^{k/s},
\end{equation}
where
$$U_1=\oint |\grF_b^\bftau(\bfalp;\eta)|^{2u+2}\d\bfalp$$
and
$$U_2=\oint |\grF_b^\bftau(\bfalp;\eta)^2\grf_{kb}(\bfalp;\zet)^{2s}|\d\bfalp .$$
On considering the underlying Diophantine system, it follows from Lemma \ref{lemma3.1} that
$$U_1\le \oint |\grf_b(\bfalp;\eta)|^{2s+2k}\d\bfalp \ll J_{s+k}(X/M^b).$$
Thus, on recalling the definition (\ref{3.4}), we find that
\begin{align*}
\oint |\grf_{kb}(\bfalp;\zet)^{2k}\grF_b^\bftau(\bfalp;\eta)^{2u}|\d\bfalp &\ll (J_{s+k}(X/M^b))^{1-k/s}(I_{b,kb}^\bftau(X;\eta,\zet))^{k/s}\notag \\
&\ll (J_{s+k}(X/M^b))^{1-k/s}(I_{b,kb}(X))^{k/s}.
\end{align*}
Finally, on substituting the latter estimate into (\ref{6.9}), the conclusion of the lemma is immediate.
\end{proof}

Before proceeding further, we pause to extract a crude but simple bound for $K_{a,b}(X)$ of value when $b$ is large.

\begin{lemma}\label{lemma6.1b}
Suppose that $a$ and $b$ are integers with $0\le a<b\le \tet^{-1}$. Then
$$\ldbrack K_{a,b}(X)\rdbrack \ll X^{\eta_{s+k}+\del}(M^{b-a})^{\frac{1}{2}k(k+1)}.$$
\end{lemma}

\begin{proof} Consider fixed integers $\xi$ and $\eta$ with $1\le \xi\le p^a$ and $1\le \eta\le p^b$, and $k$-tuples $\bfsig,\bftau\in \Sig_k$. On considering the underlying Diophantine system and applying H\"older's inequality, we deduce from (\ref{3.5}) that
\begin{align*}
K_{a,b}^{\bfsig,\bftau}(X;\xi,\eta)&\le \oint |\grf_a(\bfalp;\xi)^{2k}\grf_b(\bfalp;\eta)^{2s}|\d \bfalp \\
&\le \Bigl( \oint |\grf_a(\bfalp;\xi)|^{2s+2k}\d\bfalp \Bigr)^{k/(s+k)}\Bigl( \oint |\grf_b(\bfalp;\eta)|^{2s+2k}\d\bfalp \Bigr)^{s/(s+k)}.
\end{align*}
In view of the hypothesis $b\le \tet^{-1}$, we therefore deduce from Lemma \ref{lemma3.1} that
$$K_{a,b}(X)\le (J_{s+k}(X/M^a))^{k/(s+k)}(J_{s+k}(X/M^b))^{s/(s+k)}.$$
Consequently, on recalling (\ref{3.9}) and (\ref{3.11}), it follows that
\begin{align*}
\ldbrack K_{a,b}(X)\rdbrack &\ll \frac{X^\del \left((X/M^a)^{k/(s+k)}(X/M^b)^{s/(s+k)}\right)^{2s+2k-\frac{1}{2}k(k+1)+\eta_{s+k}}}{(X/M^b)^{2s}(X/M^a)^{2k-\frac{1}{2}k(k+1)}}\\
&\ll X^{\eta_{s+k}+\del}(M^{b-a})^{\frac{1}{2}k(k+1)s/(s+k)}.
\end{align*}
The conclusion of the lemma is now immediate.
\end{proof}

By substituting the estimate supplied by Lemma \ref{lemma5.3} into the conclusion of Lemma \ref{lemma6.1}, we obtain the basic iterative relation.

\begin{lemma}\label{lemma6.2}
Suppose that $a$ and $b$ are integers with $0\le a<b\le \frac{2}{3}(k\tet)^{-1}$. Put $H=\lceil \frac{1}{2}(k-1)b\rceil$. Then there exists an integer $h$, with $0\le h<H$, having the property that
\begin{align*}
\ldbrack K_{a,b}(X)\rdbrack \ll &\, X^\del M^{-7kh/4}(X/M^b)^{\eta_{s+k}(1-k/s)}\ldbrack K_{b,kb+h}(X)\rdbrack^{k/s}\\
&+M^{-kH/(3s)}(X/M^b)^{\eta_{s+k}}.
\end{align*}
\end{lemma}

\begin{proof} On recalling (\ref{3.9}), it follows from Lemma \ref{lemma6.1} that
\begin{equation}\label{6.10}
\ldbrack K_{a,b}(X)\rdbrack \ll (M^b)^{2s}(M^a)^{2k-\frac{1}{2}k(k+1)}M^{\frac{1}{2}k(k-1)(b+a)}(M^{kb-a})^kT_1^{1-k/s}T_2^{k/s},
\end{equation}
where
$$T_1=\frac{J_{s+k}(X/M^b)}{X^{2s+2k-\frac{1}{2}k(k+1)}}\quad \text{and}\quad T_2=\frac{I_{b,kb}(X)}{X^{2s+2k-\frac{1}{2}k(k+1)}}.$$
But in view of (\ref{3.2}), one has
\begin{equation}\label{6.11}
T_1\ll (M^{-b})^{2s+2k-\frac{1}{2}k(k+1)}(X/M^b)^{\eta_{s+k}+\del}.
\end{equation}
Write $H=\lceil \frac{1}{2}(k-1)b\rceil$, and note that the hypotheses of the statement of the lemma ensure that
$$kb+H\le kb+{\textstyle{\frac{1}{2}(k-1)}}b+{\textstyle{\frac{1}{2}}}\le {\textstyle{\frac{3}{2}}}kb\le \tet^{-1}.$$
Consequently, it follows from Lemma \ref{lemma5.3} that there exists an integer $h$ with $0\le h<H$ having the property that
$$T_2\ll \frac{M^{h(k-1)}K_{b,kb+h}(X)}{X^{2s+2k-\frac{1}{2}k(k+1)}}+\frac{M^{-H/2}(X/M^b)^{\eta_{s+k}}}{(M^{kb})^{2s}(M^b)^{2k-\frac{1}{2}k(k+1)}}.$$
On recalling (\ref{3.9}), we therefore see that
\begin{equation}\label{6.12}
T_2\ll (M^{-kb})^{2s}(M^{-b})^{2k-\frac{1}{2}k(k+1)}\Ome,
\end{equation}
in which we have written
$$\Ome =M^{-(2s-k+1)h}\ldbrack K_{b,kb+h}(X)\rdbrack +M^{-H/2}(X/M^b)^{\eta_{s+k}}.$$

\par Substituting (\ref{6.11}) and (\ref{6.12}) into (\ref{6.10}), we deduce that
$$\ldbrack K_{a,b}(X)\rdbrack \ll M^{\ome(a,b)}(X/M^b)^{(1-k/s)(\eta_{s+k}+\del)}\Ome^{k/s},$$
in which we have written
\begin{align*}
\ome(a,b)=&\, 2sb+(2k-{\textstyle \frac{1}{2}}k(k+1))a+{\textstyle \frac{1}{2}}k(k-1)(b+a)+k(kb-a)\\
&\,-(1-k/s)(2s+2k-{\textstyle \frac{1}{2}}k(k+1))b-(2skb+(2k-{\textstyle \frac{1}{2}}k(k+1))b)k/s.
\end{align*}
A modicum of computation reveals that $\ome(a,b)=0$, and thus we may infer that
\begin{align*}
\ldbrack K_{a,b}(X)\rdbrack \ll &\, (M^{-H/2})^{k/s}(X/M^b)^{\eta_{s+k}+\del(1-k/s)}\\
&\, +X^\del M^{-(2s-k+1)hk/s}(X/M^b)^{\eta_{s+k}(1-k/s)}\ldbrack K_{b,kb+h}(X)\rdbrack^{k/s}.
\end{align*}
The conclusion of the lemma follows on noting that $\del$ may be assumed small enough that $(X/M^b)^{\del(1-k/s)}\ll M^{kH/(6s)}$, and further that the assumptions $s\ge k^2$ and $k\ge 2$ together imply that $2s-k+1\ge \frac{7}{4}s$.
\end{proof}

\section{The iterative process} The estimate supplied by Lemma \ref{lemma5.4} bounds $J_{s+k}(X)$ in terms of $K_{0,1}(X)$, and Lemma \ref{lemma6.2} relates $K_{a,b}(X)$, for $b>a\ge 0$, to $K_{b,kb+h}(X)$, for some integer $h$ with $0\le h\le \frac{1}{2}(b-a)$. By repeatedly applying Lemma \ref{lemma6.2}, therefore, we are able to bound $J_{s+k}(X)$ in terms of the quantity $K_{c,d}(X)$, with $c$ and $d$ essentially as large as we please. Unfortunately, this process is not particularly simple to control, largely owing to the possibility that at any point in our iteration, a value of $h$ in the expression $K_{b,kb+h}(X)$ may be forced upon us with $h>0$. This defect in our procedure may accelerate us too rapidly towards the final step of the iteration. Our goal in this section, therefore, is to control the iterative process at a fine enough level that its potential is not substantially eroded.

\begin{lemma}\label{lemma7.1} Suppose that $a$ and $b$ are integers with $0\le a<b\le \frac{2}{3}(k\tet)^{-1}$. Suppose in addition that there exist non-negative numbers $\psi$, $c$ and $\gam$, with $c\le (2s/k)^N$, for which
\begin{equation}\label{7.1}
X^{\eta_{s+k}(1+\psi \tet)}\ll X^{c\del}M^{-\gam}\ldbrack K_{a,b}(X)\rdbrack .
\end{equation}
Then, for some non-negative integer $h$ with $h\le \frac{1}{2}(k-1)b$, one has
$$X^{\eta_{s+k}(1+\psi'\tet)}\ll X^{c'\del}M^{-\gam'}\ldbrack K_{a',b'}(X)\rdbrack ,$$
where
$$\psi'=(s/k)\psi+(s/k-1)b,\quad c'=(s/k)(c+1),\quad \gam'=(s/k)\gam +{\textstyle{\frac{7}{4}}}sh,$$
$$a'=b\quad \text{and}\quad b'=kb+h.$$
\end{lemma}

\begin{proof} Since we may suppose that $c\le (2s/k)^N$ and $\del<(Ns)^{-3N}$, we have
$$c\del<s^{-2N}/3<\tet/(3s),$$
and hence $X^{c\del}<M^{1/(3s)}$. In addition, one has $M^{1/(3s)}>X^\del$. Consequently, it follows from Lemma \ref{lemma6.2} that there exists an integer $h$ with $0\le h<\lceil \frac{1}{2}(k-1)b\rceil$ with the property that
$$\ldbrack K_{a,b}(X)\rdbrack \ll M^{-k/(3s)}X^{\eta_{s+k}}+X^\del M^{-7kh/4}(X/M^b)^{(1-k/s)\eta_{s+k}}\ldbrack K_{b,kb+h}(X)\rdbrack^{k/s}.$$
In view of the hypothesised upper bound (\ref{7.1}), therefore, we deduce that
$$X^{\eta_{s+k}(1+\psi \tet)}\ll X^{\eta_{s+k}-\del}+X^{(c+1)\del}M^{-\gam-7kh/4}(X/M^b)^{(1-k/s)\eta_{s+k}}\ldbrack K_{b,kb+h}(X)\rdbrack^{k/s},$$
whence
$$X^{\eta_{s+k}(k/s+(\psi+(1-k/s)b)\tet)}\ll X^{(c+1)\del}M^{-\gam-7kh/4}\ldbrack K_{b,kb+h}(X)\rdbrack^{k/s}.$$
The conclusion of the lemma follows on raising left and right hand sides here to the power $s/k$.
\end{proof}

Repeated application of Lemma \ref{lemma7.1} provides a series of upper bounds for $\eta_{s+k}$. What remains is to ensure that the upper bound $b\le \frac{2}{3}(k\tet)^{-1}$, required by the hypotheses of the lemma, does not preclude the possibility of making many iterations.

\begin{lemma}\label{lemma7.2}
Whenever $s\ge k^2$, one has $\eta_{s+k}=0$.
\end{lemma}

\begin{proof} We may suppose that $\eta_{s+k}>0$, for otherwise there is nothing to prove. We begin by defining three sequences $(a_n)$, $(b_n)$, $(h_n)$ of non-negative integers for $0\le n\le N$. We put $a_0=0$, $b_0=1$ and $h_0=0$. Then, when $0\le n<N$, we fix any integer $h_n$ with $0\le h_n\le \frac{1}{2}(k-1)b_n$, and then define
\begin{equation}\label{7.2}
a_{n+1}=b_n\quad \text{and}\quad b_{n+1}=kb_n+h_n.
\end{equation}
Next we define the auxiliary sequences $(\psi_n)$, $(c_n)$, $(\gam_n)$ of non-negative real numbers for $0\le n\le N$ by putting $\psi_0=0$, $c_0=1$, $\gam_0=0$. Then, for $0\le n<N$, we define
\begin{align}
\psi_{n+1}&=(s/k)\psi_n+(s/k-1)b_n,\label{7.3a}\\
c_{n+1}&=(s/k)(c_n+1),\label{7.3}\\
\gam_{n+1}&=(s/k)\gam_n+{\textstyle{\frac{7}{4}}}sh_n.\label{7.3b}
\end{align}
Notice here that an inductive argument readily confirms that $c_n\le (2s/k)^n$ for $0\le n\le N$. We claim that a choice may be made for the sequence $(h_n)$ in such a manner that for $0\le n\le N$, one has
\begin{equation}\label{7.4a}
b_n<2(s/k)^n
\end{equation}
and
\begin{equation}\label{7.4}
X^{\eta_{s+k}(1+\psi_n\tet)}\ll X^{c_n\del}M^{-\gam_n}\ldbrack K_{a_n,b_n}(X)\rdbrack .
\end{equation}
When $n=0$, the validity of the relation (\ref{7.4a}) follows by definition, whilst (\ref{7.4}) is immediate from (\ref{3.9}), (\ref{3.10}) and Lemma \ref{lemma5.4}, since the latter together imply that
$$X^{\eta_{s+k}-\del}<\ldbrack J_{s+k}(X)\rdbrack \ll \ldbrack K_{0,1}(X)\rdbrack .$$

\par We prepare the ground for the treatment of larger indices $n$ with a preliminary discussion of the recurrence relations (\ref{7.2}) to (\ref{7.3b}). Observe first that when $m\ge 0$, one has
$$\gam_{m+1}-{\textstyle{\frac{7}{4}}}k^2b_{m+1}\ge \gam_{m+1}-{\textstyle{\frac{7}{4}}}sb_{m+1}=(s/k)(\gam_m-{\textstyle{\frac{7}{4}}}k^2b_m).$$
But $\gam_0-\frac{7}{4}k^2b_0=-\frac{7}{4}k^2$, and so it follows by induction that when $0\le m\le N$, one has
\begin{equation}\label{7.5}
\gam_m\ge {\textstyle{\frac{7}{4}}}k^2(b_m-(s/k)^m).
\end{equation}

\par Suppose now that the desired conclusions (\ref{7.4a}) and (\ref{7.4}) have been established for the index $n<N$. Then as a consequence of (\ref{7.4a}) one has $kb_n\tet <k(s/k)^{n-N-1}<\frac{2}{3}$, whence $b_n<\frac{2}{3}(k\tet)^{-1}$. We may therefore apply Lemma \ref{lemma7.1} to deduce from (\ref{7.4}) that there exists a non-negative integer $h$, with $h\le\frac{1}{2}(k-1)b_n$, for which one has the upper bound
\begin{equation}\label{7.6}
X^{\eta_{s+k}(1+\psi'\tet)}\ll X^{c'\del}M^{-\gam'}\ldbrack K_{a',b'}(X)\rdbrack ,
\end{equation}
where
\begin{equation}\label{7.6a}
a'=b_n=a_{n+1},\quad b'=kb_n+h,
\end{equation}
\begin{align}
\psi'&=(s/k)\psi_n+(s/k-1)b_n=\psi_{n+1},\notag \\
c'&=(s/k)(c_n+1)=c_{n+1},\notag \\
\gam'&=(s/k)\gam_n+{\textstyle{\frac{7}{4}}}sh.\label{7.6b}
\end{align}

\par Let us suppose, if possible, that $b'\ge 2(s/k)^{n+1}$. The relations (\ref{7.6a}) and (\ref{7.6b}) then combine with (\ref{7.5}) to show that
\begin{align}
\gam'&=(s/k)\gam_n+{\textstyle{\frac{7}{4}}}s(b'-kb_n)\notag \\
&\ge (s/k)(\gam_n-{\textstyle{\frac{7}{4}}}k^2b_n)+{\textstyle{\frac{7}{4}}}k^2b'\notag \\
&\ge {\textstyle{\frac{7}{4}}}k^2(b'-(s/k)^{n+1})\ge {\textstyle{\frac{7}{8}}}k^2b'.\label{7.6c}
\end{align}
But $b'=kb_n+h\le \frac{3}{2}kb_n<\tet^{-1}$, and so it follows from Lemma \ref{lemma6.1b} that
\begin{equation}\label{7.6d}
\ldbrack K_{a',b'}(X)\rdbrack \ll X^{\eta_{s+k}+\del}(M^{b'})^{\frac{1}{2}k(k+1)}.
\end{equation}
Thus, on substituting (\ref{7.6c}) and (\ref{7.6d}) into (\ref{7.6}), we arrive at the upper bound
$$X^{\eta_{s+k}(1+\psi_{n+1}\tet)}\ll X^{\eta_{s+k}+(c_{n+1}+1)\del}(M^{b'})^{\frac{1}{2}k(k+1)-\frac{7}{8}k^2}.$$
We now recall that $c_{n+1}\le (2s/k)^{n+1}$, and thus confirm that $X^{(c_{n+1}+1)\del}<M^{1/2}$. In this way, we obtain the upper bound $X^{\eta_{s+k}\psi_{n+1}\tet}\ll M^{-1/2}$. Since $\psi_{n+1}$ and $\tet$ are both positive, we are forced to conclude that $\eta_{s+k}<0$, contradicting our opening hypothesis. The assumption that $b'\ge 2(s/k)^{n+1}$ is therefore untenable, and so we must in fact have $b'<2(s/k)^{n+1}$. We take $h_{n+1}$ to be the integer $h$ at hand, so that $b'=b_{n+1}$ and $\gam'=\gam_{n+1}$, and thereby we obtain the desired conclusion that (\ref{7.4a}) and (\ref{7.4}) hold with $n$ replaced by $n+1$. This completes the present inductive step.\par

At this point, we have confirmed the validity of the relations (\ref{7.4a}) and (\ref{7.4}) for $0\le n\le N$. We next bound the sequences occurring in (\ref{7.4}) so as to extract a suitable conclusion. The bound $c_n<(n+1)(s/k)^n$ is readily confirmed by induction from (\ref{7.3}), and the lower bound $\gam_n\ge 0$ already suffices for our purposes at this stage. In addition, the relation (\ref{7.2}) plainly implies that $b_n\ge k^n$, whence from (\ref{7.3a}) we deduce that for $s\ge k^2$, one has
$$\psi_{n+1}\ge k\psi_n+(k-1)k^n,$$
and by induction this delivers the lower bound $\psi_n\ge n(k-1)k^{n-1}$. Finally, we find from (\ref{7.4a}) that $b_N\tet <k/s<1$, whence $b_N<\tet^{-1}$. Making use of Lemma \ref{lemma6.1b}, therefore, we find from (\ref{7.4}) that
\begin{equation}\label{7.a}
X^{\eta_{s+k}(1+\psi_N\tet)}\ll X^{\eta_{s+k}+(c_N+1)\del}(M^{b_N})^{\frac{1}{2}k(k+1)}\ll X^{\eta_{s+k}+k^2}.
\end{equation}
But since $\tet=\frac{1}{2}(k/s)^{N+1}$, it follows that
$$\eta_{s+k}\le \frac{k^2}{\psi_N\tet}\le \frac{2k^2(s/k)^{N+1}}{N(k-1)k^{N-1}}.$$
It is at this point only that we restrict $s$ to be $k^2$, and thus we obtain the upper bound $\eta_{k(k+1)}\le 2k^4/N$. But we are at liberty to take $N$ as large as we please in terms of $k$, and thus $\eta_{k(k+1)}$ can be made arbitrarily small. We are therefore forced to conclude that in fact $\eta_{k(k+1)}=0$. But then, as in the discussion of the opening paragraph of \S3, we may conclude that $\eta_s=0$ whenever $s\ge k(k+1)$. This completes the proof of the lemma.
\end{proof}

We have now reached the crescendo of this opus, for in view of (\ref{2.01}) and (\ref{2.02}), the conclusion of Lemma \ref{lemma7.2} already establishes Theorem \ref{theorem1.1}.\par

A perusal of the proof of Lemma \ref{lemma7.2} might give the impression that it is critical to the success of our iterative process that $s=k^2$, and that the method is inherently unstable. This notion is, however, mistaken. If one were to have $s>\frac{3}{2}k^2$, then one easily reaches the conclusion that $\eta_{s+k}=0$ simply by comparing the rates of growth of $\psi_n$ and $b_n$ in the above argument. 
Such a procedure can also be adapted, with care, to the range $s>\frac{5}{4}k^2$. It is only when $k^2\le s\le \frac{5}{4}k^2$ that the behaviour of the sequences $(b_n)$ and $(\psi_n)$, depending as they do on $(h_n)$, become so difficult to control. The restriction to the case $s=k^2$ should, therefore, be seen rather as a simplifying manoeuvre rather than an inescapable mandate. 

\section{Estimates of Weyl type} The derivation of our upper bounds for Weyl sums, and the application of these estimates to analyse the distribution of polynomials modulo $1$, is easily accomplished by applying Theorem \ref{theorem1.1} within results familiar from the literature. We are therefore concise in our discussion of the associated arguments.

\begin{proof}[The proof of Theorem \ref{theorem1.5}] With the hypotheses of the statement of Theorem \ref{theorem1.5}, it follows from \cite[Theorem 5.2]{Vau1997} that for each natural number $s$, one has
$$f_k(\bfalp;X)\ll (J_{s,k-1}(2X)X^{\frac{1}{2}k(k-1)}(q^{-1}+X^{-1}+qX^{-j}))^{1/(2s)}\log (2X).$$
But from Theorem \ref{theorem1.1} it follows that when $s=k(k-1)$, one has
$$J_{s,k-1}(2X)\ll X^{2s-\frac{1}{2}k(k-1)+\eps},$$
and thus
$$f_k(\bfalp;X)\ll X^{1+\eps}(q^{-1}+X^{-1}+qX^{-j})^{1/(2k(k-1))}.$$
As we shall find in \S9 below, when $s\ge k^2-k+1$, one has also the $\eps$-free upper bound $J_{s,k-1}(X)\ll X^{2s-\frac{1}{2}k(k-1)}$, and in like manner this delivers the estimate
$$f_k(\bfalp;X)\ll X(q^{-1}+X^{-1}+qX^{-j})^{1/(2k^2-2k+2)}\log (2X).$$
\end{proof}

\begin{proof}[The proof of Theorem \ref{theorem1.6}]
One may establish Theorem \ref{theorem1.6} by applying the argument underlying the proofs of \cite[Theorems 4.3 and 4.4]{Bak1986}. Let $\eps$ be a sufficiently small positive number. We begin by putting $\tau=1/(4k(k-1))$ and $A=X^{1-\tau+\eps}$, and then observe that Theorem \ref{theorem1.1} shows that one may replace $\tet$ by $\eps$ in the case $l=k$ of \cite[Theorem 4.3]{Bak1986}. In this way, we find that the hypotheses of the statement of Theorem \ref{theorem1.6} imply that there exist coprime pairs of integers $q_j$, $b_j$ $(2\le j\le k)$ such that
$$q_j\ge 1,\quad |q_j\alp_j-b_j|\le X^{\eps-j}(X/A)^{2k(k-1)}\ (2\le j\le k),$$
and such that the least common multiple $q_0$ of $q_2,\ldots ,q_k$ satisfies
$$q_0\le X^\eps (X/A)^{2k(k-1)}.$$
Notice here that
$$X^\eps (X/A)^{2k(k-1)}\le X^\eps (X^{\tau-\eps})^{2k(k-1)}<X^{\frac{1}{2}-3\eps}.$$
Write $r$ for $q_0$, and $v_j$ for $b_jq_0/q_j$ $(2\le j\le k)$. Then one has
$$|r\alp_j-v_j|\le X^{2\eps -j}(X/A)^{4k(k-1)}\le X^{1-j}/(4k^4)\quad (2\le j\le k).$$
Next, denote by $d$ the greatest common divisor $d=(r,v_2,\dots ,v_k)$. Then, with the hypotheses of the statement of Theorem \ref{theorem1.6}, it is a consequence of \cite[Lemma 4.6]{Bak1986} that there is a natural number $t$ with $t\le 2k^2$ such that
\begin{align*}
trd^{-1}&\le (X/A)^kX^{3k\eps}\\
t|r\alp_j-v_j|d^{-1}&\le (X/A)^kX^{3k\eps-j}\ (2\le j\le k)\\
\|trd^{-1}\alp_1\|&\le (X/A)^kX^{3k\eps -1}.
\end{align*}
But $(X/A)^k=X^{k\tau-k\eps}$, and so whenever $\del>k\tau+2k\eps$ one may conclude that there exist integers $q,a_1,\dots ,a_k$ such that
$$1\le q\le X^\del\quad \text{and}\quad |q\alp_j-a_j|\le X^{\del-j}\quad (1\le j\le k).$$
Since we have supposed $\eps$ to be sufficiently small, the same conclusion follows whenever $\del>k\tau$, and so the proof of Theorem \ref{theorem1.6} is complete.
\end{proof}

\begin{proof}[The proof of Theorem \ref{theorem1.7}.]
We may apply the argument of the proof of \cite[Theorem 4.4]{Bak1986}, substituting the modifications available from Theorem \ref{theorem1.6} above and its proof. Let $\del$ be a positive number. Suppose that $P\ll X$ and $(MXP^{-1})^{4k(k-1)}\le X^{1-\del}$. Then we find that when
$$\sum_{m=1}^M|f_k(m\bfalp;X)|\ge P,$$
then there exist integers $y,u_1,\dots ,u_k$ such that
$$1\le y\le M(MXP^{-1})^kX^\eps\quad \text{and}\quad |y\alp_j-u_j|\le (MXP^{-1})^kX^{\eps-j}\ (1\le j\le k).$$
From here, as in the proof of \cite[Theorem 4.5]{Bak1986}, the remaining part of our argument is straightforward. If one has
\begin{equation}\label{8.0}
\min_{1\le n\le X}\|\alp_1n+\dots +\alp_kn^k\|>X^{\del-\tau(k)},
\end{equation}
then with $M=[X^{\tau(k)-\del}]+1$, one obtains the lower bound
$$\sum_{m=1}^M|f_k(m\bfalp;X)|>{\textstyle{\frac{1}{6}}}X.$$
The above discussion then shows that there exists a natural number $y$ such that
$$y\ll M^{k+1}X^\eps \ll X^{(k+1)\tau(k)+\eps}\quad \text{and}\quad \|y\alp_j\|\ll X^{k\tau(k)-j+\eps}\quad (1\le j\le k).$$
Thus we find that $y\le X$ and that
$$\|\alp_1y+\ldots +\alp_ky^k\|\le \sum_{j=1}^kX^{j-1}\|y\alp_j\|\ll X^{k\tau(k)-1+\eps}<X^{-\tau(k)}.$$
This upper bound contradicts our earlier hypothesis (\ref{8.0}), and thus we are forced to conclude that
$$\min_{1\le n\le X}\|\alp_1n+\ldots +\alp_kn^k\|\le X^{\del-\tau(k)}.$$
This completes the proof of Theorem \ref{theorem1.7}.
\end{proof}

\section{Tarry's problem, and related topics} Our discussion of Tarry's problem follows a familiar path. Let $s$ be a natural number with $s\le k^3$, and define $\rho(\bfh)$ to be the number of integral solutions of the system of equations
$$\sum_{i=1}^sx_i^j=h_j\quad (1\le j\le k),$$
with $1\le \bfx\le X$. In addition, let $\sig(\bfg)$ denote the number of integral solutions of the system of equations
$$\sum_{i=1}^sx_i^j=g_j\quad (1\le j\le k+1),$$
with $1\le \bfx\le X$. Observe that
$$\rho(\bfh)=\sum_{1\le g_{k+1}\le sX^{k+1}}\sig(\bfh,g_{k+1}).$$
Consequently, if for all values of $\bfh$ one were to have $\sig(\bfh,g_{k+1})\ne 0$ only for a set $\calA(\bfh)$ of values of $g_{k+1}$ of cardinality at most $t$, then it would follow from Cauchy's inequality that
$$\rho(\bfh)^2\le \Bigl( \sum_{\substack{1\le g_{k+1}\le sX^{k+1}\\ g_{k+1}\in \calA(\bfh)}}\sig(\bfh,g_{k+1})\Bigr)^2\le \text{card}(\calA(\bfh))\sum_{1\le g_{k+1}\le sX^{k+1}}\sig(\bfh,g_{k+1})^2.$$
If such were the case, then one would have
\begin{align*}
J_{s,k}(X)&=\sum_{1\le h_1\le sX}\dots \sum_{1\le h_k\le sX^k}\rho(\bfh)^2\\
&\le t\sum_{1\le h_1\le sX}\dots \sum_{1\le h_k\le sX^k}\sum_{1\le g_{k+1}\le sX^{k+1}}\sig(\bfh,g_{k+1})^2=tJ_{s,k+1}(X).
\end{align*}
What we have shown is that when $X$ is sufficiently large, and $J_{s,k}(X)>tJ_{s,k+1}(X)$, then there exists a choice of $\bfh$ such that there are more than $t$ choices for $g_{k+1}$ with 
$\sig(\bfh,g_{k+1})>0$. There therefore exists a solution of the system
$$\sum_{i=1}^sx_{i1}^j=\sum_{i=1}^sx_{i2}^j=\ldots =\sum_{i=1}^sx_{it}^j\quad (1\le j\le k),$$
in which the sums $\sum_{i=1}^sx_{il}^{k+1}$ $(1\le l\le t)$ take distinct values. We have therefore shown that whenever
\begin{equation}\label{9.1}
J_{s,k}(X)>tJ_{s,k+1}(X),
\end{equation}
then $W(k,t)\le s$.\par

We seek to establish that for some positive number $\del$, one has
\begin{equation}\label{9.2}
J_{s,k+1}(X)\ll X^{2s-\frac{1}{2}k(k+1)-\del}.
\end{equation}
In view of the lower bound (\ref{1.5}), an estimate of this quality suffices to establish (\ref{9.1}). But from Theorem \ref{theorem1.1}, one has
$$J_{s,k+1}(X)\ll X^{2s-\frac{1}{2}(k+1)(k+2)+\eps}$$
whenever $s\ge (k+1)(k+2)$. Moreover, the estimate
$$J_{k+2,k+1}(X)\ll X^{k+2}$$
follows from \cite{VW1997}, and indeed earlier results would suffice here. By interpolating via H\"older's inequality, therefore, we find that when $s$ is an integer with $k+2\le s\le (k+1)(k+2)$, then
$$J_{s,k+1}(X)\ll X^{2s-\frac{1}{2}(k+1)(k+2)+\eta_s+\eps},$$
where
\begin{align*}
\eta_s&=((k+1)(k+2)-s)\left( \frac{{\textstyle \frac{1}{2}}(k+1)(k+2)-(k+2)}{(k+1)(k+2)-(k+2)}\right)\\
&={\textstyle \frac{1}{2}}(1-1/k)((k+1)(k+2)-s).
\end{align*}
It follows that the condition (\ref{9.2}) is satisfied whenever
$${\textstyle \frac{1}{2}}(1-1/k)((k+1)(k+2)-s)<k+1,$$
or equivalently,
$$(k+1)(k+2)-s<2k\left( \frac{k+1}{k-1}\right) =2k+4+\frac{4}{k-1}.$$
We deduce that (\ref{9.2}) holds whenever $s\ge (k+1)(k+2)-2k-4$, and hence $W(k,t)\le k^2+k-2$. This completes the proof of Theorem \ref{theorem1.3}.\par

There may be some scope for improvement in the upper bound presented in Theorem \ref{theorem1.3} by exploiting the sharpest bounds available from Vinogradov's mean value theorem for smaller moments (see \cite{Woo1992}, \cite{Woo1994}, \cite{For2002} and \cite{BK2010}). In this way, one might hope to improve even the coefficient of $k$ in the upper bound for $W(k,h)$, though not that of $k^2$.

\begin{proof}[The proof of Theorem \ref{theorem1.2}]
Let $s$ and $k$ be natural numbers with $k\ge 3$ and $s\ge k^2+k+1$, and let $X$ be a positive number sufficiently large in terms of $s$ and $k$. We follow the argument of the proof of \cite[Theorem 3]{Woo1996}. When $1\le q\le X^{1/k}$, $1\le a_j\le q$ $(1\le j\le k)$ and $(q,a_1,\ldots ,a_k)=1$, define the major arc $\grM(q,\bfa)$ by
$$\grM(q,\bfa)=\{ \bfalp \in [0,1)^k\,:\, |q\alp_j-a_j|\le X^{1/k-j}\ (1\le j\le k)\}.$$
It is not hard to check that the arcs $\grM(q,\bfa)$ are disjoint. Let $\grM$ denote the union of the major arcs $\grM(q,\bfa)$ with $q$ and $\bfa$ as above, and define the minor arcs $\grm$ by $\grm=[0,1)^k\setminus \grM$. Then from (\ref{1.2}) we have
\begin{equation}\label{9.3}
J_{s,k}(X)=\int_\grM |f(\bfalp;X)|^{2s}\d\bfalp +\int_\grm |f(\bfalp;X)|^{2s}\d\bfalp .
\end{equation}

\par We first bound the contribution of the minor arcs. As a consequence of Theorem \ref{theorem1.6}, one finds that
$$\sup_{\bfalp \in \grm}|f(\bfalp;X)|\le X^{1-\tau+\eps},$$
where $\tau^{-1}=4k(k-1)$. Then it follows from Theorem \ref{theorem1.1} that
\begin{align}
\int_\grm |f(\bfalp;X)|^{2s}\d\bfalp &\ll \left( \sup_{\bfalp \in \grm}|f(\bfalp;X)|\right)^{2s-2k^2-2k}\oint |f(\bfalp;X)|^{2k^2+2k}\d\bfalp \notag \\
&\ll (X^{1-\tau+\eps})^{2s-2k^2-2k}X^{\frac{3}{2}k(k+1)+\eps}\notag \\
&\ll X^{2s-\frac{1}{2}k(k+1)-1/(3k^2)}.\label{9.4}
\end{align}

\par Next we discuss the major arc contribution. When $\bfalp \in \grM(q,\bfa)\subseteq \grM$, write
$$V(\bfalp;q,\bfa)=q^{-1}S(q,\bfa)I(\bfalp-\bfa/q;X),$$
where
$$S(q,\bfa)=\sum_{r=1}^qe((a_1r+\dots +a_kr^k)/q)$$
and
$$I(\bfbet;X)=\int_0^Xe(\bet_1\gam+\dots +\bet_k\gam^k)\d\gam .$$
In addition, define the function $V(\bfalp)$ to be $V(\bfalp;q,\bfa)$ when $\bfalp\in \grM(q,\bfa)\subseteq \grM$, and to be zero otherwise. Then the argument concluding \cite[\S3]{Woo1996} shows that
\begin{align}
\int_\grM &|f(\bfalp;X)|^{2s}\d\bfalp -\int_\grM |V(\bfalp)|^{2s}\d\bfalp \notag \\
&\ll X^{1+2/k}\Bigl( \oint |f(\bfalp;X)|^{2s-2}\d\bfalp +\oint |V(\bfalp)|^{2s-2}\d\bfalp \Bigr) .\label{9.5}
\end{align}

\par When $\bfalp \in \grM(q,\bfa)\subseteq \grM$, one has $(q,a_1,\ldots ,a_k)=1$ and $|q\alp_j-a_j|\le X^{1/k-j}$ $(1\le j\le k)$. Then it follows from \cite[Theorems 7.1 and 7.3]{Vau1997} that when $\bfalp \in \grM(q,\bfa)\subseteq \grM$, one has
$$V(\bfalp)\ll Xq^\eps (q+|q\alp_1-a_1|X+\dots +|q\alp_k-a_k|X^k)^{-1/k}.$$
Consequently, one finds that when $t\ge \frac{1}{2}k(k+1)$, one has
$$\int_\grM |V(\bfalp)|^{2t}\d\bfalp \ll X^{2t}WZ,$$
where
$$W=\sum_{1\le q\le X^{1/k}}\sum_{a_1=1}^q\dots \sum_{a_k=1}^q(q^{\eps -1/k})^{2t}$$
and
$$Z=\prod_{j=1}^k\int_0^{X^{1/k-j}}(1+\bet_jX^j)^{-2t/k^2}\d \bet_j .$$
But since $2t\ge k(k+1)$, we obtain the upper bounds
\begin{equation}\label{9.6}
W\ll X^{1/(3k)}\sum_{q=1}^\infty q^{-5/4}\ll X^{1/(3k)}
\end{equation}
and
\begin{equation}\label{9.7}
Z\ll \prod_{j=1}^k\int_0^\infty (1+\bet_jX^j)^{-1-1/k}\d\bet_j \ll X^{-\frac{1}{2}k(k+1)}.
\end{equation}
Thus, in particular, we deduce that when $s\ge k^2+k+1$, then
$$\int_\grM |V(\bfalp)|^{2s-2}\d\bfalp \ll X^{2s-2-\frac{1}{2}k(k+1)+1/(3k)}.$$
In combination with Theorem \ref{theorem1.1}, this leads from (\ref{9.5}) to the asymptotic relation
\begin{equation}\label{9.8}
\int_\grM |f(\bfalp;X)|^{2s}\d\bfalp -\int_\grM |V(\bfalp)|^{2s}\d\bfalp \ll X^{2s-\frac{1}{2}k(k+1)-1/(3k)}.
\end{equation}

\par The argument employed in deriving (\ref{9.6}) and (\ref{9.7}) is readily adapted to show that the singular series $\grS(s,k)$ defined in (\ref{1.9b}), and the singular integral $\grJ(s,k)$ defined in (\ref{1.9c}), both converge absolutely, and that
$$\int_\grM |V(\bfalp)|^{2s}\d\bfalp =\grS(s,k)\grJ(s,k)+O(X^{2s-\frac{1}{2}k(k+1)-1/(3k)}).$$
The asymptotic formula claimed implicitly in Theorem \ref{theorem1.2} now follows by substituting (\ref{9.4}) and (\ref{9.8}) into (\ref{9.3}). This completes the proof of Theorem \ref{theorem1.2}.
\end{proof}

As essentially was observed by Vaughan, one must have both $\grS(s,k)\gg 1$ and $\grJ(s,k)\gg 1$ (see the conclusion of \cite[\S7.3]{Vau1997}). For otherwise one would have
$$\oint |f(\bfalp;X)|^{2s}\d\bfalp =o(X^{2s-\frac{1}{2}k(k+1)}),$$
which contradicts the elementary lower bound (\ref{1.5}).\par

An argument similar to that employed in the proof of Theorem \ref{theorem1.2} delivers an asymptotic formula for the number of solutions of a more general diagonal Diophantine system. When $s$ and $k$ are natural numbers, and $a_{ij}$ are integers for $1\le i\le k$ and $1\le j\le s$, we write
$$\phi_i(\bfx)=\sum_{j=1}^sa_{ij}x_j^i\quad (1\le i\le k),$$
and we consider the Diophantine system
\begin{equation}\label{9.A}
\phi_i(\bfx)=0\quad (1\le i\le k).
\end{equation}
We write $N(B)$ for the number of integral solutions of the system (\ref{9.A}) with $|\bfx|\le B$. We next define the (formal) real and $p$-adic densities associated with the system (\ref{9.A}), and here we follow Schmidt \cite{Sch1985}. When $L>0$, define
$$\lam_L(\eta)=\begin{cases} L(1-L|\eta|),&\text{when $|\eta|\le L^{-1}$,}\\
0,&\text{otherwise.}\end{cases}$$
We then put
$$\mu_L=\int_{|\bfxi|\le 1}\prod_{i=1}^k\lam_L(\phi_i(\bfxi))\d\bfxi .$$
The limit $\sig_\infty=\lim_{L\rightarrow \infty}\mu_L$, when it exists, is called the {\it real density}. Meanwhile, given a natural number $q$, we write
$$M(q)=\text{card}\{ \bfx\in (\dbZ/q\dbZ)^s\,:\, \phi_i(\bfx)\equiv 0\pmod{q}\ (1\le i\le k)\}.$$
For each prime number $p$, we then put
$$\sig_p=\lim_{H\rightarrow \infty}p^{H(k-s)}M(p^H),$$
provided that this limit exists, and refer to $\sig_p$ as the {\it $p$-adic density}. 

\begin{theorem}\label{theorem9.1} Let $s$ and $k$ be natural numbers with $k\ge 3$ and $s\ge 2k^2+2k+1$. Suppose that $a_{ij}$ $(1\le i\le k,\, 1\le j\le s)$ are non-zero integers. Suppose in addition that the system of equations (\ref{9.A}) possess non-singular real and $p$-adic solutions, for each prime number $p$. Then one has
$$N(B)\sim \sig_\infty \Bigl(\prod_p\sig_p\Bigr) B^{s-\frac{1}{2}k(k+1)}.$$
In particular, the system (\ref{9.A}) satisfies the Hasse Principle.
\end{theorem}

We will not offer any details of the proof here, the argument following in most respects that of the proof of Theorem \ref{theorem1.2}. We note only that the system (\ref{9.A}), if singular, is easily shown to have a singular locus of affine dimension at most $k-1$, which is harmless in the analysis. We note also that the restriction that $a_{ij}\ne 0$ $(1\le i\le k,\, 1\le j\le s)$ may be largely removed by elaborating on the basic argument. We emphasise that the most striking feature of Theorem \ref{theorem9.1} is that such a conclusion cannot possibly hold when $s<\frac{1}{2}k(k+1)$. Thus, for the very first time for a system of diagonal equations of higher degree, we have an asymptotic formula in which the number of variables is just four times the best possible result. Hitherto, the number of variables required to achieve a successful analysis would be roughly $2\log k$ times the best possible result, a factor which becomes arbitrarily large as $k$ increases.\par

We turn our attention next to the Hilbert-Kamke problem, a generalisation of Waring's problem considered first by Hilbert \cite{Hil1909}. When $n_1,\ldots,n_k$ are natural numbers, let $R_{s,k}(\bfn)$ denote the number of solutions in natural numbers $\bfx$ of the system of equations
\begin{equation}\label{9.C}
\sum_{i=1}^sx_i^j=n_j\quad (1\le j\le k).
\end{equation}
Put
$$X=\max_{1\le j\le k}n_j^{1/j},$$
and then write
$$\calJ_{s,k}(\bfn)=\int_{\dbR^k}I(\bfbet;1)^se(-\bet_1n_1/X-\dots -\bet_kn_k/X^k)\d\bfbet$$
and
$$\calS_{s,k}(\bfn)=\sum_{q=1}^\infty \sum_{\substack{1\le \bfa\le q\\ (q,a_1,\ldots ,a_k)=1}}(q^{-1}S(q,\bfa))^se(-(a_1n_1+\dots +a_kn_k)/q).$$
The local solubility conditions associated with the system (\ref{9.C}) are quite subtle, and we refer the reader to \cite{Ark1984} for a discussion of the conditions under which real and $p$-adic solutions may be expected to exist for the system (\ref{9.C}). It is easy to see, however, that the conditions
$$n_k^{j/k}\le n_j\le s^{1-j/k}n_k^{j/k}\quad (1\le j\le k),$$
are needed. One also finds that $p$-adic solubility is not assured without at least $2^k$ variables.

\begin{theorem}\label{theorem9.2}
Let $s$ and $k$ be natural numbers with $k\ge 3$ and $s\ge 2k^2+2k+1$. Suppose that the natural numbers $n_1,\ldots ,n_k$ are sufficiently large in terms of $s$ and $k$. Put $X=\max_{1\le j\le k}n_j^{1/j}$. Suppose in addition that the system (\ref{9.C}) has non-singular real and $p$-adic solutions. Then one has
$$R_{s,k}(\bfn)\sim \calJ_{s,k}(\bfn)\calS_{s,k}(\bfn)X^{s-\frac{1}{2}k(k+1)}.$$
\end{theorem}

We refer the reader to \cite{Ark1984}, \cite{Mit1986}, \cite{Mit1987} for the many details associated with a successful treatment of this problem. The technology available at the time of writing of the latter papers made necessary the constraint $s\ge (4+o(1))k^2\log k$ in place of the lower bound $s\ge 2k^2+2k+1$ in Theorem \ref{theorem9.2}. Our observation here is that a successful local-global analysis is now available via the circle method when the number of variables grows like $2k^2+2k+1$, only a factor of $4$ away from what is likely to be best possible.

\section{The asymptotic formula in Waring's problem} The proof of Theorem \ref{theorem1.4} would be routine were our goal the less precise bound $\Gtil(k)\le 2k^2+2k+1$. Saving four additional variables requires some discussion which hints at possible new strategies for transforming estimates for $J_{s,k}(X)$ into upper bounds for $\Gtil(k)$. En route we also improve some old estimates of Hua \cite{Hua1965}.\par

Write
$$g(\alp)=\sum_{1\le x\le X}e(\alp x^k),$$
and when $s\in \dbN$, define
$$I_s(X)=\int_0^1|g(\alp)|^{2s}\d\alp .$$
Then on considering the underlying Diophantine system, one has
\begin{align*}
I_s(X)&=\sum_{|h_1|\le sX}\dots \sum_{|h_{k-1}|\le sX^{k-1}}\oint |f(\bfalp;X)|^{2s}e(-h_1\alp_1-\dots -h_{k-1}\alp_{k-1})\d\bfalp \\
&\ll X^{\frac{1}{2}k(k-1)}\oint |f(\bfalp;X)|^{2s}\d\bfalp =X^{\frac{1}{2}k(k-1)}J_{s,k}(X).
\end{align*}
Thus we obtain the classical bound
\begin{equation}\label{10.1}
I_s(X)\ll X^{2s-k+\eta_s+\eps}.
\end{equation}

\par Ford \cite{For1995} obtained a bound potentially sharper, valid for each natural number $m$ with $1\le m\le k$, and $s\ge \frac{1}{2}m(m-1)$, which is tantamount to
$$I_s(X)\ll X^{2s-k+\eta_{s,m}^*+\eps},$$
where $\eta_{s,m}^*=\frac{1}{m}\eta_{s-\frac{1}{2}m(m-1)}$. A little later, this conclusion was obtained independently by Ustinov \cite{Ust1998}. Owing to the efficiency of Theorem \ref{theorem1.1}, this estimate proves to be no sharper than that provided by (\ref{10.1}), at least in applications to the asymptotic formula in Waring's problem. Instead we offer a very modest refinement of (\ref{10.1}). The idea underlying this refinement is related to one first shown to the author by Bob Vaughan in the first year of the author's Ph.D. studies, in 1988.

\begin{lemma}\label{lemma10.1}
For each natural number $s$, one has
$$I_s(X)\ll X^\eps (X^{2s-k-1+\eta_{s,k}}+X^{2s-k+\eta_{s,k-1}}).$$
\end{lemma}

\begin{proof} Define the exponential sum $F(\bfbet)=F_k(\bfbet;X)$ by
$$F(\bfbet)=\sum_{1\le x\le X}e(\bet_kx^k+\bet_{k-2}x^{k-2}+\dots +\bet_1x).$$
Thus, to be precise, the argument of the exponentials in $F(\bfbet)$ is a polynomial of degree $k$ in which the coefficient of the monomial of degree $k-1$ is zero. Also, define $\Ups_k(X;h)$ to be the number of integral solutions of the Diophantine system
\begin{align}
\sum_{i=1}^s(x_i^j-y_i^j)&=0\quad (1\le j\le k,\, j\ne k-1),\notag \\
\sum_{i=1}^s(x_i^{k-1}-y_i^{k-1})&=h,\label{10.2}
\end{align}
with $1\le \bfx,\bfy\le X$. Then on considering the underlying Diophantine system, one finds that
\begin{equation}\label{10.3}
\oint |F(\bfbet)|^{2s}\,\d\bfbet =\sum_{|h|\le sX^{k-1}}\Ups_k(X;h).
\end{equation}

\par By applying an integer shift $z$ to the variables in the system (\ref{10.2}), we find that $\Ups_k(X;h)$ counts the number of integral solutions of the Diophantine system
\begin{align*}
\sum_{i=1}^s((x_i-z)^j-(y_i-z)^j)&=0\quad (1\le j\le k,\, j\ne k-1),\\
\sum_{i=1}^s((x_i-z)^{k-1}-(y_i-z)^{k-1})&=h,
\end{align*}
with $1+z\le \bfx,\bfy\le X+z$. But by applying the Binomial Theorem, we find that $\bfx,\bfy$ satisfies this system of equations if and only if
\begin{align}
\sum_{i=1}^s(x_i^j-y_i^j)&=0\quad (1\le j\le k-2) \notag\\
\sum_{i=1}^s(x_i^{k-1}-y_i^{k-1})&=h,\label{10.4} \\
\sum_{i=1}^s(x_i^k-y_i^k)&=khz.\notag
\end{align}
If we restrict the shifts $z$ to lie in the interval $1\le z\le X$, then we see that an upper bound for $\Ups_k(X;h)$ is given by the number of integral solutions of the system (\ref{10.4}) with $1\le \bfx,\bfy\le 2X$. On considering the underlying Diophantine system, we therefore deduce from (\ref{10.3}) that
\begin{align}
\oint |F(\bfbet)|^{2s}\d\bfbet &\le \sum_{|h|\le sX^{k-1}}\oint |f(\bfalp ;2X)|^{2s}e(-(kz\alp_k+\alp_{k-1})h)\d\bfalp \notag \\
&\ll X^{-1}\sum_{1\le z\le X}\oint |f(\bfalp;2X)|^{2s}\min \{ X^{k-1},\| kz\alp_k+\alp_{k-1}\|^{-1}\} \d\bfalp \notag \\
&=X^{-1}\oint |f(\bfalp ;2X)|^{2s}\Psi(\alp_k,\alp_{k-1})\d\bfalp ,\label{10.5}
\end{align}
where we have written
$$\Psi(\alp_k,\alp_{k-1})=\sum_{1\le z\le X}\min \{ X^{k-1},\|kz\alp_k+\alp_{k-1}\|^{-1}\} .$$

\par Suppose that $\alp_k\in\dbR$, and that $b\in \dbZ$ and $r\in \dbN$ satisfy $(b,r)=1$ and $|\alp_k-b/r|\le r^{-2}$. Then it follows from \cite[Lemma 3.2]{Bak1986}\footnote{We note that the strict inequality $|\alp_k-b/r|<r^{-2}$ imposed by Baker is unnecessary in the proof of \cite[Lemma 3.2]{Bak1986}} that
\begin{align}
\Psi(\alp_k,\alp_{k-1})&\ll (X^{k-1}+r\log (2r))(X/r+1)\notag \\
&\ll X^k(X^{-1}+r^{-1}+rX^{-k})(\log (2r)).\label{10.7}
\end{align}
Applying a standard transference principle (compare Exercise 2 of \cite[\S2.8]{Vau1997}), it follows that
\begin{equation}\label{10.8}
\Psi(\alp_k,\alp_{k-1})\ll X^{k+\eps}(X^{-1}+(r+X^k|r\alp_k-b|)^{-1}+(r+X^k|r\alp_k-b|)X^{-k}).
\end{equation}

\par We now return to consider the relation (\ref{10.5}). Let $\grm$ denote the set of real numbers $\alp\in [0,1)$ having the property that whenever $q\in \dbN$ and $\|q\alp\|\le X^{1-k}$, then $q>X$. Also, let $\grM$ denote the complementary set $[0,1)\setminus \grm$. By Dirichlet's theorem on Diophantine approximation, whenever $\alp_k\in \grm$, there exists $q\in \dbN$ with $q\le X^{k-1}$ such that $\|q\alp \|\le X^{1-k}$. From the definition of $\grm$, one must have $q>X$, and hence it follows from (\ref{10.7}) that
$$\sup_{\alp_k\in \grm}\Psi(\alp_k,\alp_{k-1})\ll X^{k-1+\eps}.$$
Thus we deduce from (\ref{1.2}) that
\begin{align*}
\int_{\grm\times [0,1)^{k-1}}|f(\bfalp;2X)|^{2s}\Psi(\alp_k,\alp_{k-1})\d\bfalp &\ll X^{k-1+\eps}\oint |f(\bfalp;2X)|^{2s}\d\bfalp \\
&\ll X^{k-1+\eps}J_{s,k}(2X).
\end{align*}
Substituting this conclusion into (\ref{10.5}), we see that
\begin{align}
\oint |F(\bfbet)|^{2s}\d\bfbet \ll &\,X^{k-2+\eps}J_{s,k}(2X) \notag\\
&\,+X^{-1}\int_{\grM\times [0,1)^{k-1}}|f(\bfalp ;2X)|^{2s}\Psi(\alp_k,\alp_{k-1})\d\bfalp .\label{10.9}
\end{align}

Let $\grM(q,a)$ denote the set of real numbers $\alp_k\in [0,1)$ with $|q\alp_k-a|\le X^{1-k}$. Then $\grM$ is the union of the sets $\grM(q,a)$ with $0\le a\le q\le X$ and $(a,q)=1$. From (\ref{10.8}) it follows that when $\alp_k\in \grM(q,a)\subseteq \grM$, one has
$$\Psi (\alp_k,\alp_{k-1})\ll X^{k-1+\eps}+X^{k+\eps}(q+X^k|q\alp_k-a|)^{-1}.$$
Define the function $\Phi(\tet)$ for $\tet\in \grM$ by putting
$$\Phi(\tet)=(q+X^k|q\tet-a|)^{-1}$$
when $\tet\in \grM(q,a)\subseteq \grM$. Then we deduce from (\ref{10.9}) that
\begin{equation}\label{10.10}
\oint |F(\bfbet)|^{2s}\d\bfbet \ll X^{k-2+\eps}J_{s,k}(2X)+X^{k-1+\eps}\calT,
\end{equation}
where
$$\calT=\int_\grM \Phi (\alp_k)\oint |f(\bfbet,\alp_k;2X)|^{2s}\d\bfbet \d\alp_k.$$
From Br\"udern \cite[Lemma 2]{Bru1988}, we find that
\begin{align*}
\int_\grM \Phi(\alp_k)&|f(\bfbet,\alp_k;2X)|^{2s}\d\alp_k \\
&\ll X^{\eps-k}\Bigl( X\int_0^1|f(\bfbet,\alp_k;2X)|^{2s}\d\alp_k+|f(\bfbet,0;2X)|^{2s}\Bigr) ,
\end{align*}
and hence
\begin{align*}
\calT&\ll X^{\eps-k}\Bigl( X\oint |f_k(\bfalp;2X)|^{2s}\d\bfalp +\oint |f_{k-1}(\bfbet ;2X)|^{2s}\d\bfbet \Bigr)\\
&\ll X^{\eps-k}(XJ_{s,k}(2X)+J_{s,k-1}(2X)).
\end{align*}
Consequently, from (\ref{10.10}) we conclude that
\begin{equation}\label{10.11}
\oint |F(\bfbet)|^{2s}\d\bfbet \ll X^{k-2+\eps}J_{s,k}(2X)+X^{\eps-1}J_{s,k-1}(2X).
\end{equation}

\par Next we observe that, on considering the underlying Diophantine system, one has
$$I_s(X)=\sum_{|h_1|\le sX}\dots \sum_{|h_{k-2}|\le sX^{k-2}}R(X;\bfh),$$
where $R(X;\bfh)$ denotes the number of integral solutions of the system
\begin{align*}
\sum_{i=1}^s(x_i^j-y_i^j)&=h_j\quad (1\le j\le k-2)\\
\sum_{i=1}^s(x_i^k-y_i^k)&=0,
\end{align*}
with $1\le \bfx,\bfy\le X$. Thus, again considering the underlying Diophantine system, we obtain the upper bound
\begin{align*}
I_s(X)&\ll \sum_{|h_1|\le sX}\dots \sum_{|h_{k-2}|\le sX^{k-2}}\oint |F(\bfbet)|^{2s}e(-\bet_1h_1-\dots -\bet_{k-2}h_{k-2})\d\bfbet \\
&\ll X^{\frac{1}{2}(k-1)(k-2)}\oint |F(\bfbet)|^{2s}\d\bfbet .
\end{align*}
In view of (\ref{10.11}), we therefore arrive at the estimate
\begin{align*}
I_s(X)&\ll X^{\frac{1}{2}(k+1)(k-2)+\eps}J_{s,k}(2X)+X^{\frac{1}{2}(k-1)(k-2)-1+\eps}J_{s,k-1}(2X)\\
&\ll X^{2s-k-1+\eta_{s,k}+\eps}+X^{2s-k+\eta_{s,k-1}+\eps}.
\end{align*}
This completes the proof of the lemma.
\end{proof}

From Theorem \ref{theorem1.1}, we have $\eta_{s,k-1}=0$ for $s\ge k(k-1)$. By H\"older's inequality, moreover, one finds from Theorem \ref{theorem1.1} that
\begin{align*}
\oint |f_k(\bfalp;X)|^{2k^2+2k-4}\d\bfalp &\le \Bigl( \oint |f(\bfalp;X)|^{2k^2+2k}\d\bfalp \Bigr)^{1-2/(k^2+k)}\\
&\ll \left( X^{\frac{3}{2}(k^2+k)+\eps}\right)^{1-2/(k^2+k)}\\
&\ll X^{\frac{3}{2}(k^2+k)-3+\eps}.
\end{align*}
Consequently, one has $\eta_{s,k}\le 1$ for $s\ge k^2+k-2$. Then by Lemma \ref{lemma10.1}, we obtain the following corollary to Lemma \ref{lemma10.1}.

\begin{corollary}\label{corollary10.2}
When $s\ge k^2+k-2$, one has $I_s(X)\ll X^{2s-k+\eps}$.
\end{corollary}

Having prepared the ground, the proof of Theorem \ref{theorem1.4} is now swift. Consider a large integer $n$, put $X=[n^{1/k}]$, and recall the definition of the sets of arcs $\grm$ and $\grM$ from the proof of Lemma \ref{lemma10.1}. From Corollary \ref{corollary10.2} and Weyl's inequality (see \cite[Lemma 2.4]{Vau1997}), one finds that when $t\ge 2k^2+2k-3$, one has
\begin{align*}
\int_\grm g(\alp)^te(-n\alp)\d\bfalp &\ll \left( \sup_{\alp\in \grm}|g(\alp)|\right)^{t-(2k^2+2k-4)}\int_0^1|g(\alp)|^{2k^2+2k-4}\d\bfalp \\
&\ll (X^{1-2^{1-k}+\eps})^{t-(2k^2+2k-4)}X^{(2k^2+2k-4)-k}\\
&\ll X^{t-k-2^{-k}}.
\end{align*}
Notice here, of course, that we could have employed the conclusion of Theorem \ref{theorem1.5} in place of Weyl's inequality. Meanwhile, the methods of \cite[\S4.4]{Vau1997} show that, under the same conditions on $t$, one has
$$\int_\grM g(\alp)^te(-n\alp)\,d\alp \sim \frac{\Gam (1+1/k)^t}{\Gam(t/k)}\grS_{t,k}(n)n^{t/k-1}+o(n^{t/k-1}),$$
where $\grS_{t,k}(n)$ is defined as in (\ref{1.10a}). Thus we deduce that for $t\ge 2k^2+2k-3$, one has
\begin{align*}
R_{t,k}(n)&=\int_\grM g(\alp)^te(-n\alp)\d\alp +\int_\grm g(\alp)^te(-n\alp)\d\alp \\
&=\frac{\Gam(1+1/k)^t}{\Gam(t/k)}\grS_{t,k}(n)n^{t/k-1}+o(n^{t/k-1}),
\end{align*}
whence $\Gtil(k)\le 2k^2+2k-3$. This completes the proof of Theorem \ref{theorem1.4}.\par

We take this opportunity to point out that L.-K. Hua investigated the problem of bounding the least integer $C_k$ such that, whenever $s\ge C_k$, one has
$$\oint |f_k(\bfalp;X)|^s\d\bfalp \ll X^{s-\frac{1}{2}k(k+1)+\eps},$$
and likewise the least integer $S_k$ such that, whenever $s\ge S_k$, one has
$$\oint |F_k(\bfbet ;X)|^s\d\bfbet \ll X^{s-\frac{1}{2}(k^2-k+2)+\eps},$$
pursuing in particular the situation for smaller values of $k$. His arguments involve a clever application of Weyl differencing in a style that we would describe in the single equation situation as underlying Hua's lemma. In Chapter 5 of \cite{Hua1965}, one finds tables recording the upper bounds
$$C_3\le 16,\quad C_4\le 46,\quad C_5\le 110,\dots $$
and
$$S_3\le 10,\quad S_4\le 32,\quad S_5\le 86,\dots .$$
The conclusion of Theorem \ref{theorem1.1} shows that $C_k\le 2k(k+1)$, an upper bound superior to the conclusions of Hua for $k\ge 4$. Meanwhile, as a consequence of the estimate (\ref{10.11}), one obtains the estimate contained in the following theorem.

\begin{theorem}\label{theorem10.1} Suppose that $k\ge 3$ and $s\ge k^2+k-2$. Then one has
$$\oint |F_k(\bfbet;X)|^{2s}\d\bfbet \ll X^{2s-\frac{1}{2}(k^2-k+2)+\eps}.$$
\end{theorem}

\begin{proof}
The discussion leading to Corollary \ref{corollary10.2} shows that $\eta_{s,k-1}=0$ for $s\ge k(k-1)$ and $\eta_{s,k}\le 1$ for $s\ge k^2+k-2$. The desired conclusion is therefore immediate from (\ref{3.8}), (\ref{3.11}) and (\ref{10.11}).
\end{proof}

Thus we have $S_k\le 2k^2+2k-4$, an upper bound superior to those of Hua for $k\ge 5$.

\section{A heuristic argument} We take the opportunity in this section to discuss a heuristic argument which delivers the bound
\begin{equation}\label{11.1}
J_{s,k}(X)\ll X^{2s-\frac{1}{2}k(k+1)+\eps}
\end{equation}
for $s\ge \frac{1}{2}k(k+1)$. In view of the lower bound (\ref{1.5}), of course, the bound (\ref{11.1}) cannot hold for $s<\frac{1}{2}k(k+1)$, so is in a strong sense best possible.\par

Our starting point is a heuristic interpretation of Lemma \ref{lemma6.1}. In the course of the proof of Lemma \ref{lemma6.1}, a critical role is played by the interpretation of the system of equations (\ref{6.3}) by means of the implied congruences (\ref{6.5}). In some sense, for each fixed choice of $\bfy$ in (\ref{6.5}), the conclusion of Lemma \ref{lemma4.1} indicates that there are at most $k!p^{\frac{1}{2}k(k-1)(a+b)}$ possible choices for $\bfx$ with $1\le \bfx\le p^{kb}$ and $\bfx\equiv \bfxi\pmod{p^{a+1}}$ for some $\bfxi\in \Xi_a(\xi)$. This is transformed via Cauchy's inequality into the statement that, with a compensating factor $k!p^{\frac{1}{2}k(k-1)(a+b)}$, the variables in (\ref{6.3}) are constrained by the additional congruence relations $\bfx\equiv \bfy\pmod{p^{kb}}$. Such an interpretation is embodied in the relation (\ref{6.8}).\par

An alternative interpretation, which we emphasise is heuristic in nature and not a statement of fact, is that, by relabelling variables if necessary, the congruences (\ref{6.5}) essentially amount in (\ref{6.3}) to the constraint $x_j\equiv y_j\pmod{p^{jb}}$ $(1\le j\le k)$, with an additional compensating factor of $k!p^{\frac{1}{2}k(k-1)a}$. Indeed, one can prove the initial statement that $x_j\equiv y_j\pmod{p^b}$ $(1\le j\le k)$ with precisely this compensating factor. Then, by fixing the variables $x_1,y_1$, and considering the system (\ref{6.5}) with $2\le j\le k$, one might suppose that a corresponding constraint $x_j\equiv y_j\pmod{p^{2b}}$ $(2\le j\le k)$ might be imposed. Then, by fixing the variables $x_2,y_2$, and considering the system (\ref{6.5}) with $3\le j\le k$, one seeks a corresponding constraint $x_j\equiv y_j\pmod{p^{3b}}$ $(3\le j\le k)$, and so on. Such a heuristic implies a new relation to replace (\ref{6.8}) of the shape
$$K^{\bfsig,\bftau}_{a,b}(X;\xi,\eta)\ll M^{\frac{1}{2}k(k-1)a}\sum_{\substack{1\le \zet_1\le p^b\\ \zet_1\equiv \xi\mmod{p^a}}}\dots \sum_{\substack{1\le \zet_k\le p^{kb}\\ \zet_k\equiv \xi\mmod{p^a}}}\calI(\bfzet),$$
where
$$\calI(\bfzet)=\oint \Bigl( \prod_{i=1}^k|\grf_{ib}(\bfalp;\zet_i)|^2\Bigr) |\grF^\bftau_b(\bfalp;\eta)|^{2u}\d\bfalp .$$
Such an assertion at least carries the weight of correctly accounting for the number of available residue classes, though of course one cannot hope for the implied degree of independence to be true in anything but an average sense.\par

From here, an application of H\"older's inequality leads to the bound
\begin{equation}\label{11.2}
K^{\bfsig,\bftau}_{a,b}(X;\xi,\eta)\ll M^{\frac{1}{2}k(k-1)a}\Bigl( \prod_{i=1}^kM^{ib-a}\Tet_{ib,b}(X;\eta)^{1/k}\Bigr),
\end{equation}
where
$$\Tet_{c,b}(X;\eta)=\max_{1\le \zet\le p^c}\oint |\grf_c(\bfalp;\zet)^{2k}\grF_b^\bftau (\bfalp;\eta)^{2u}|\d\bfalp .$$
A further application of H\"older's inequality shows as in (\ref{6.9c}) that
$$\Tet_{c,b}(X;\eta)\ll (J_{s+k}(X/M^b))^{1-k/s}(I_{b,c}(X))^{k/s},$$
and thus we find from (\ref{11.2}) that
$$K_{a,b}(X)\ll M^{\frac{1}{2}k(k-1)(a+b)+k(b-a)}(J_{s+k}(X/M^b))^{1-k/s}\prod_{i=1}^k(I_{b,ib}(X))^{1/s}.$$
Each mean value $I_{b,ib}(X)$ may be conditioned via Lemma \ref{lemma5.3}, and thus one deduces as in Lemma \ref{lemma6.2} that there exist integers $h_1,\dots ,h_k$, none too large in terms of $b$, with the property that
\begin{align}
\ldbrack K_{a,b}(X)\rdbrack \ll &\, M^{-k/(3s)}(X/M^b)^{\eta_{s+k}}\notag \\
&\, +X^\del (X/M^b)^{\eta_{s+k}(1-k/s)}\prod_{i=1}^kM^{-7h_i/4}\ldbrack K_{b,ib+h_i}(X)\rdbrack^{1/s}\label{11.4}.
\end{align}

\par It is (\ref{11.4}) which represents the critical step in our iteration. Starting from the relation
$$X^{\eta_{s+k}-\del}<\ldbrack J_{s+k}(X)\rdbrack \ll \ldbrack K_{0,1}(X)\rdbrack ,$$
one may apply (\ref{11.4}) successively to bound $X^{\eta_{s+k}-\del}$ in terms first of the $s$ expressions of the shape $\ldbrack K_{1,i+h_i}(X)\rdbrack^{1/s}$ $(1\le i\le s)$, then of $s^2$ expressions of the shape $\ldbrack K_{b,ib+h'_i}(X)\rdbrack^{1/s^2}$, and so on. This iteration may be analysed in a manner very similar to that used in the proof of Lemma \ref{lemma7.2}, though the complexity is now increased substantially. The important feature is the number of iterations taken before the exponents $ib+h_i$ occurring in (\ref{11.4}) become large in terms of $\tet$. In the argument of the proof of Lemma \ref{lemma7.2}, one finds that at the $n$th iteration, the relevant exponents have size roughly $k^n$. From the relation (\ref{11.4}), one obtains an explosively growing tree of chains of relations, with the exponents $b_n$ increasing from one step to the next by a factor close to $1,2,\ldots,k-1$ or $k$. When one considers the set of all chains, one finds that almost all possible chains have the property that the exponent $b_n$ grows on average like $(\frac{1}{2}(k+1))^n$. In order to see this, observe that if $l_1,\ldots ,l_n$ are the factors at each step of one possible chain, then by the Arithmetic-Geometric Mean inequality, one has
$$l_1\cdots l_n\le \Bigl( \frac{l_1+\dots +l_n}{n}\Bigr)^n.$$
If one randomly chooses $l_1,\dots ,l_n$ from $\{1,2,\ldots ,k\}$ with equal probability, then almost all values of $(l_1+\dots +l_n)/n$ will be concentrated towards the mean of $\{1,2,\ldots ,k\}$, which is $\frac{1}{2}(k+1)$. This is a consequence of the Central Limit Theorem. In this way, one sees that the number of steps permitted before the iteration begins to exhaust its usefulness is roughly $N$ if we take $\tet=\frac{1}{2}((k+1)/2)^{-N-1}$ at the outset in place of $\tet=\frac{1}{2}k^{-N-1}$. Note that the latter is indeed the value that we chose for $\tet$ in \S3 when $s=k^2$.\par

We are led now to a relation of similar shape to (\ref{7.a}), but replaced now by
\begin{equation}\label{11.4a}
X^{\eta_{s+k}(1+(s/k-1)(s/k)^{N-1}\tet)}\ll X^{\eta_{s+k}+k^2}.
\end{equation}
Note here that we have made use of the growth rate of the exponents $\psi_n$ from \S7, with scale factor $s/k$. Thus, when $s>\frac{1}{2}k(k+1)$, since now we have $\tet=\frac{1}{2}((k+1)/2)^{-N-1}$, we find that
$$(s/k-1)(s/k)^{N-1}\tet \gg \Bigl( \frac{s}{k(k+1)/2}\Bigr)^N,$$
which tends to infinity as $N$ tends to infinity. In particular, on taking $N$ sufficiently large, the relation (\ref{11.4a}) implies that $\eta_{s+k}=0$.\par

The above heuristic shows that when $s>\frac{1}{2}k(k+1)$, then one has
\begin{equation}\label{11.5}
J_{s+k}(X)\ll X^{2s-\frac{1}{2}k(k+1)+\eps}.
\end{equation}
One might complain that this fails to prove that the relation (\ref{11.5}) holds for $s=\frac{1}{2}k(k+1)$. Apart from anything else, on the face of it, the integer $s$ needs to be a multiple of $k$ in our treatment, so that one may need to require that $s\ge \frac{1}{2}k(k+3)$. But this issue may be circumvented. For this, one reinterprets the methods of this paper in the form of fractional moments of exponential sums along the lines of the author's work \cite{Woo1995a} on breaking classical convexity in Waring's problem. This was, in fact, the author's original approach to Theorem \ref{theorem1.1}, and feasible with sufficient effort. Such would permit the proof of (\ref{11.5}) with $s=\frac{1}{2}k(k+1)+\nu$, for any positive number $\nu$. But then an application of H\"older's inequality shows that (\ref{11.5}) holds with $s=\frac{1}{2}k(k+1)$ with the positive number $\eps$ bounded above by $\nu$. Taking $\nu$ sufficiently small completes the heuristic proof.\par

A final word is in order concerning the value of such a heuristic argument. A more sweeping heuristic of classical nature asserts that one should expect square-root cancellation in $f_k(\alp;X)$ when one subtracts the expected major arc approximation, and this leads to the conjectured estimate (\ref{11.5}) for $s\ge \frac{1}{2}k(k+1)$. This amounts to the assumption of very significant global rigid structure within the mean value $J_{s,k}(X)$. Our heuristic in this section also amounts to a structural assumption, but now of a rather weak congruential variety. This is, most assuredly, an unproven assumption, but a relatively modest one of local type. Thus one can say, at least, that the conjectured estimate (\ref{11.5}) for $s\ge \frac{1}{2}k(k+1)$ now rests on only a relatively mild assumption.

\bibliographystyle{amsbracket}

\begin{thebibliography}{18}

\bibitem{Ark1984}
G. I. Arkhipov, \emph{On the Hilbert-Kamke problem}, Izv. Akad. Nauk SSSR Ser. Mat. \textbf{48} (1984), 3--52.

\bibitem{ACK}
G. I Arkhipov, V. N. Chubarikov and A. A. Karatsuba, \emph{Trigonometric sums in number theory and analysis}, Walter de Gruyter, Berlin, 2004.

\bibitem{Bak1986}
R. C. Baker, \emph{Diophantine inequalities}, London Mathematical Society Monographs, vol. \textbf{1}, Oxford University Press, Oxford, 1986.

\bibitem{Bir1961}
B. J. Birch, \emph{Waring's problem in algebraic number fields}, Proc. Cambridge Philos. Soc. \textbf{57} (1961), 449--459. 

\bibitem{Bok1994}
K. D. Boklan, \emph{The asymptotic formula in Waring's problem}, Mathematika \textbf{41} (1994), 329--347.

\bibitem{BK2010}
K. D. Boklan and T. D. Wooley, \emph{On Weyl sums for smaller exponents}, Funct. Approx. Comment. Math., to appear.

\bibitem{Bru1988}
J. Br\"udern, \emph{A problem in additive number theory}, Math. Proc. Cambridge Philos. Soc. \textbf{103} (1988), 27--33.

\bibitem{CH2010}
E. Croot and D. Hart, \emph{$h$-fold sums from a set with few products}, SIAM J. Discrete Math. \textbf{24} (2010), 505--519.

\bibitem{For1995}
K. B. Ford, \emph{New estimates for mean values of Weyl sums}, Internat. Math. Res. Notices (1995), 155--171.

\bibitem{For2002}
K. B. Ford, \emph{Vinogradov's integral and bounds for the Riemann zeta function}, Proc. London Math. Soc. (3) \textbf{85} (2002), 565--633.

\bibitem{HL1922}
G. H. Hardy and J. E. Littlewood, \emph{Some problems of `Partitio Numerorum': IV. The singular series in Waring's Problem and the value of the number $G(k)$}, Math. Zeit. \textbf{12} (1922), 161--188.

\bibitem{HB1988}
D. R. Heath-Brown, \emph{Weyl's inequality, Hua's inequality, and Waring's problem}, J. London Math. Soc. (2) \textbf{38} (1988), 216--230.

\bibitem{Hil1909}
D. Hilbert, \emph{Beweis f\"ur die Darstellbarkeit der ganzen Zahlen durch eine feste Anzahlen $n^{\text{ter}}$ Potenzen (Waringsches Problem)}, Math. Ann. \textbf{67} (1909), 281--300.

\bibitem{Hua1938}
L.-K. Hua, \emph{On Tarry's problem}, Quart. J. Math. Oxford \textbf{9} (1938), 315--320.

\bibitem{Hua1949a}
L.-K. Hua, \emph{Improvement of a result of Wright}, J. London Math. Soc. \textbf{24} (1949), 157--159.

\bibitem{Hua1949b}
L.-K. Hua, \emph{An improvement of Vinogradov's mean-value theorem and several applications}, Quart. J. Math. Oxford \textbf{20} (1949), 48--61.

\bibitem{Hua1965}
L.-K. Hua, \emph{Additive theory of prime numbers}, American Math. Soc., Providence, RI, 1965.

\bibitem{Kar1973}
A. A. Karatsuba, \emph{The mean value of the modulus of a trigonometric sum}, Izv. Akad. Nauk SSSR \textbf{37} (1973), 1203--1227.

\bibitem{Lin1943}
Yu. V. Linnik, \emph{On Weyl's sums}, Mat. Sbornik (Rec. Math.) \textbf{12} (1943), 28--39.

\bibitem{Mit1986}
D. A. Mit'kin, \emph{Estimate for the number of summands in the Hilbert-Kamke problem}, Mat. Sbornik (N.S.) \textbf{129} (1986), 549--577.

\bibitem{Mit1987}
D. A. Mit'kin, \emph{Estimate for the number of summands in the Hilbert-Kamke problem, II}, Mat. Sbornik (N.S.) \textbf{132} (1987), 345--351.

\bibitem{Par2005}
S. T. Parsell, \emph{A generalization of Vinogradov's mean value theorem}, Proc. London Math. Soc. (3) \textbf{91} (2005), 1--32.

\bibitem{Par2009}
S. T. Parsell, \emph{On the Bombieri-Korobov estimate for Weyl sums}, Acta Arith. \textbf{138} (2009), 363--372.

\bibitem{RS2000}
O. Robert and P. Sargos, \emph{Un th\'eor\`eme de moyenne pour les sommes d'exponentielles. Application \`a l'in\'egalit\'e de Weyl}, Publ. Inst. Math. (Beograd) (N.S.) \textbf{67} (2000), 14--30.

\bibitem{Sch1985}
W. M. Schmidt, \emph{The density of integer points on homogeneous varieties}, Acta Math. \textbf{154} (1985), 243--296.

\bibitem{Ste1975}
S. B. Stechkin, \emph{On mean values of the modulus of a trigonometric sum}, Trudy Mat. Inst. Steklov \textbf{134} (1975), 283--309.

\bibitem{Ust1998}
A. V. Ustinov, \emph{On the number of summands in the asymptotic formula for the number of solutions of the Waring equation}, Mat. Zametki \textbf{64} (1998), 285--296.

\bibitem{Vau1986a}
R. C. Vaughan, \emph{On Waring's problem for cubes}, J. Reine Angew. Math. \textbf{365} (1986), 122--170.

\bibitem{Vau1986b}
R. C. Vaughan, \emph{On Waring's problem for smaller exponents, II}, Mathematika \textbf{33} (1986), 6--22.

\bibitem{Vau1997}
R. C. Vaughan, \emph{The Hardy-Littlewood method}, 2nd edition, Cambridge University Press, Cambridge, 1997. 

\bibitem{VW1997}
R. C. Vaughan and T. D. Wooley, \emph{A special case of Vinogradov's mean value theorem}, Acta Arith. \textbf{79} (1997), 193--204.

\bibitem{Vin1935}
I. M. Vinogradov, \emph{New estimates for Weyl sums}, Dokl. Akad. Nauk SSSR \textbf{8} (1935), 195--198.

\bibitem{Vin1947}
I. M. Vinogradov, \emph{The method of trigonometrical sums in the theory of numbers}, Trav. Inst. Math. Steklov \textbf{23}, Moscow, 1947.

\bibitem{Vin1958}
I. M. Vinogradov, \emph{A new estimate of the function $\zet (1+it)$}, Izv. Akad. Nauk SSSR Ser. Mat. \textbf{22} (1958), 161--164.

\bibitem{Wal1963}
A. Z. Walfisz, \emph{Weylsche Exponentialsummen in der neueren Zahlentheorie}, Math. Forsch. \textbf{15}, Berlin, 1963.

\bibitem{Woo1992a}
T. D. Wooley, \emph{Large improvements in Waring's problem}, Ann. of Math. (2) \textbf{135} (1992), 131--164.

\bibitem{Woo1992}
T. D. Wooley, \emph{On Vinogradov's mean value theorem}, Mathematika \textbf{39} (1992), 379--399.

\bibitem{Woo1993}
T. D. Wooley, \emph{On Vinogradov's mean value theorem, II}, Michigan Math. J. \textbf{40} (1993), 175--180.

\bibitem{Woo1994}
T. D. Wooley, \emph{Quasi-diagonal behaviour in certain mean value theorems of additive number theory}, J. Amer. Math. Soc. \textbf{7} (1994), 221--245.

\bibitem{Woo1995}
T. D. Wooley, \emph{New estimates for Weyl sums}, Quart. J. Math. Oxford (2) \textbf{46} (1995), 119--127.

\bibitem{Woo1995a}
T. D. Wooley, \emph{Breaking classical convexity in Waring's problem: sums of cubes and quasi-diagonal behaviour}, Invent. Math. \textbf{122} (1995), 421--451.

\bibitem{Woo1996}
T. D. Wooley, \emph{Some remarks on Vinogradov's mean value theorem and Tarry's problem}, Monatsh. Math. \textbf{122} (1996), 265--273.

\bibitem{Woo2011}
T. D. Wooley, \emph{The asymptotic formula in Waring's problem}, in preparation.

\bibitem{Wri1948}
E. M. Wright, \emph{The Prouhet-Lehmer problem}, J. London Math. Soc. \textbf{23} (1948), 279--285.

\end{thebibliography}
\providecommand{\bysame}{\leavevmode\hbox to3em{\hrulefill}\thinspace}

\end{document}